\documentclass[11pt]{amsart}
\usepackage{amsmath,amsfonts,xargs,enumitem}
\usepackage[utf8]{inputenc}
\usepackage{graphicx} 
\usepackage{amsmath, amsthm, amssymb, amsfonts}
\usepackage{tkz-tab}
\usepackage{theoremref}
\usepackage{thmtools}
\usepackage{hyperref}
\usepackage{comment}
\usepackage{color}
\usepackage{mathtools}
\usepackage{subcaption}
\mathtoolsset{showonlyrefs}
\usepackage{lmodern}
\hypersetup{colorlinks=true,urlcolor=black, pdftitle=''FunOfAlpha''}
\usepackage{pgfplots}
\pgfplotsset{compat=1.18}

\newcommand{\R}{{\mathbb R}}

\theoremstyle{plain}

\newtheorem{theorem}{Theorem}[section]
\newtheorem{lemma}[theorem]{Lemma}

\newtheorem{corollary}[theorem]{Corollary}
\newtheorem{prop}[theorem]{Proposition}

\theoremstyle{definition}

\newtheorem{defn}[theorem]{Definition}

\theoremstyle{remark}

\newtheorem{remark}{Remark}

\title[On $\alpha$-convolutions]{Integral inequalities for $\alpha$-convolutions\\of $\alpha$-concave functions}
\author[M. Madiman]{Mokshay Madiman}
\address[M. Madiman]{University of Delaware, Newark, DE, 19716, USA}
\email{madiman@udel.edu}
\author[A. Manui]{Auttawich Manui}
\address[A. Manui]{Kent State University, Kent, OH 44242, USA}
\email{amanui@kent.edu}
\author[B. Zawalski]{Bart\l{}omiej Zawalski}
\address[B. Zawalski]{Case Western Reserve University, Cleveland, OH 44106, USA}
\email{bxz428@case.edu}
\author[A. Zvavitch]{Artem Zvavitch}
\address[A. Zvavitch]{Kent State University, Kent, OH 44242, USA}
\email{zvavitch@math.kent.edu}


\begin{document}

\begin{abstract}
Classical sumset inequalities originating in additive combinatorics admit geometric analogues for convex bodies in finite-dimensional real vector spaces, as recently developed by Fradelizi and two of the authors. We develop integral analogues for geometric $\alpha$-concave functions under $\alpha$-convolutions, an operation that arises naturally in the ``geometrization of probability'' program. In particular, we establish a Pl\"unnecke--Ruzsa-type inequality, as well as sharp analogues of sum-difference and Ruzsa triangle inequalities,  for $\alpha$-convolutions of $\alpha$-concave functions. We also prove a sharp Rogers--Shephard-type inequality and characterize its equality cases for $\alpha$-concave functions, bridging the log-concave case studied by Alonso-Guti\'errez, Gonz\'alez-Merino, Jim\'enez, and Villa and the quasi-concave case studied by Colesanti.
\end{abstract}

\maketitle

\section{Introduction}

\subsection{Functional liftings of sumset inequalities 
in \texorpdfstring{$\R^n$}{Rn}}

The translation of concepts from convex geometry into analytic frameworks, namely functional analysis and information theory, has become an active area of research. The analytic viewpoint extends classical geometric inequalities to functions (whose ``size'' is measured by the integral) or to probability measures (whose ``size'' is measured by the entropy). These so-called integral and entropic liftings replace geometric statements involving sets by functional or probabilistic analogues. Prominent examples are the Pr\'ekopa--Leindler inequality and the entropy power inequality, which are analytic and probabilistic counterparts of the Brunn--Minkowski inequality. We refer the reader to the surveys  \cite{G-02,Mi-08, S-14, MMX-17, Col17} for detailed discussions of these liftings.

If $A+B=\{a+b: a \in A, b \in B\}$ is the Minkowski sum of nonempty compact subsets $A, B \subset \mathbb{R}^n$, and $|\cdot|$ denotes the (Lebesgue) measure on $\R^n$, the Brunn--Minkowski inequality asserts that
\begin{equation}
    \label{eq:Brunn–Minkowski inequality}
    |A+B|^{1 / n} \geq|A|^{1 / n}+|B|^{1 / n}.
\end{equation}
The functional lifting of the Brunn--Minkowski inequality is given by the Pr\'ekopa--Leindler inequality \cite{P-73,L-72}. For integrable functions $f,g :\R^n \to [0,\infty)$ and any $\lambda \in [0,1]$, the Pr\'ekopa--Leindler inequality asserts that
\begin{equation}
    \label{eq:pli}
    \int \left(\lambda \cdot f\right) \star \left((1-\lambda) \cdot g\right)
    \geq
    \left( \int f \right)^{\lambda}
    \left( \int g \right)^{1-\lambda},
\end{equation}
where
$
    (f \star g) (x) =\sup_y f(y) g(x-y)
$
is the Asplund product, and $(\lambda \cdot f) (x) = f(x/\lambda)^\lambda$.
Observe that when $f = 1_A$ and  $g = 1_B$ for compact sets $A, B$,  the integrand on the left is $1_{\lambda A + (1-\lambda)B}$, and then the Pr\'ekopa--Leindler inequality \eqref{eq:pli} reduces to an equivalent form of the Brunn--Minkowski inequality \eqref{eq:Brunn–Minkowski inequality}. The Pr\'ekopa--Leindler inequality is central to understanding the class of log-concave functions.

Throughout this paper, we work with a one-parameter family of function classes that generalizes log-concave functions.

\begin{defn} \label{def:C_alpha}
Let $-\infty \leq \alpha \leq 0$ be fixed. A function $f: \mathbb{R}^n\to \mathbb{R}_+$ is $\alpha$-concave if for all $x,y$ in the support of $f$ and $\lambda \in [0,1]$,
\begin{equation}
    f(\lambda x + (1 - \lambda)y)
    \geq \left[\lambda f(x)^\alpha + (1 - \lambda)f(y)^\alpha\right]^{1/\alpha}.
\end{equation}
We write
\begin{equation*}
C_\alpha(\mathbb{R}^n)=
\left\{f: \mathbb{R}^n\to \mathbb{R}_+  \Bigg|\begin{aligned}
\,\,\, & f \text{ is upper semicontinuous and $\alpha$-concave}  , \\
&\|f\|_{\infty}=f(0)=1, \lim_{|x| \to \infty}  f(x) = 0
\end{aligned}\right\}.
\end{equation*}
\end{defn}

This definition forms a continuous bridge between indicator functions of convex sets ($\alpha = \infty$), log-concave functions ($\alpha= 0$), and quasi-concave functions ($\alpha = -\infty$), extending geometric convexity into the analytic domain. The original motivation of this definition was the Borell--Brascamp--Lieb inequality, developed by Borell~\cite{B-75} and later Brascamp and Lieb~\cite{BL-76}, whose concavity parameter naturally suggests consideration of the class of $\alpha$-concave functions.

Note that every element of $C_\alpha(\mathbb{R}^n)$ attains its maximum value $1$ at the origin. These functions are sometimes called the ``geometric
$\alpha$-concave functions''.

We also work with a one-parameter family of convolutions, called $\alpha$-convolutions, which are a generalization of the Asplund product introduced by Rotem \cite{R-13,MR-13}.
While we give the precise definition in  Section~\ref{section: alpha-concave},
for $-\infty < \alpha <0$, the $\alpha$-convolution is given informally by
\begin{equation}
    (f \star_\alpha g) (x) = \sup_{ y} (f(y)^\alpha + g(x-y)^\alpha -1)^{1/\alpha},
\end{equation}
and converges to the Asplund product as \(\alpha\to0^{-}\).
If $f(y)=0$ or $g(x-y)=0$, then the expression inside the supremum is taken to be zero.

The fact that the volumes of Minkowski sums of convex bodies are much more constrained than those of arbitrary compact bodies motivates the enormous literature on volume inequalities in convex geometry.
Our perspective, as in the literature on functional liftings in general, is that when we shift our consideration from compact sets to the
more general setting of integrable functions,
a good replacement for the notion of convexity in the functional domain is to consider $\alpha$-concave functions for some $\alpha$. Adjusting the value of $\alpha$ allows us to adjust the strength of the convexity assumption we are making on the functions under consideration.

The goal of this paper is to establish functional liftings of some fundamental sumset inequalities for convex bodies, which are expressed as inequalities for the integrals of $\alpha$-convolutions of $\alpha$-concave functions when $\alpha \leq 0$. In several cases, we are able to establish inequalities with sharp constants.

In the rest of this introduction, we outline our main results by first recalling some useful and interesting volume inequalities for convex bodies in  ${\mathbb R}^n$, and then describing their functional liftings.

\subsection{Functional liftings of Pl\"unnecke--Ruzsa-type inequalities}

The classical Pl\"unnecke--Ruzsa inequalities \cite{P-70, R-89, TV-06}, which  play a central role in additive combinatorics, were first developed in the discrete setting of abelian groups, and generalized by Ruzsa \cite{R-96} to locally compact abelian groups equipped with Haar measure. In particular, he showed that for any compact sets $A,B_1, \ldots ,B_m$ in ${\mathbb R}^n$,
\begin{equation}\label{eq:ruzvol1}
|A|^{m-1}|B_1+\ldots +B_m| \le \prod_{i=1}^m|A + B_i|.
\end{equation}
This is an immediate consequence of a slightly more general result asserting that if $|A|>0$, for every $\varepsilon >0$, there exists a nonempty compact set $A' \subset A $ with
 \begin{equation}\label{eq:ruzvol}
 |A|^m|A'+B_1+\ldots +B_m| \le (1+\varepsilon)|A'| \prod_{i=1}^m|A + B_i|.
 \end{equation}
Note that when $m=2$, this looks rather similar to
\eqref{eq:Ent_Plun-Ruz}, except for the need of transition to a subset $A'$.

On the other hand, the entropic lifting of the Brunn--Minkowski inequality is the entropy power inequality ~\cite{S-48,S-59}, which asserts that for independent random variables $X,Y$ in $\R^n$, one has
\begin{equation}
    \label{eq:EPI}
    N(X+Y) \geq N(X)+N(Y) .
\end{equation}
Here $N(X) = e^{2 h(X) / n}$ is called the entropy power of $X$,  where $h(X) = \mathbb{E} [ -\log f(X)]$ denotes the differential or Shannon-Boltzmann entropy, and $f$ is the probability density function of $X$.
The first-named author established the log-submodularity of the entropy power \cite{M-08} (see also \cite{MK-18} for variants and consequences, developed in the general setting of locally compact abelian groups). If $X,Y,Z$ are independent random variables with absolutely continuous distributions, this property states that
\begin{equation}
    \label{eq:Ent_Plun-Ruz}
    N(X)N(X+Y+Z) \leq N(X+Y)N(X+Z).
\end{equation}
It is natural to think of the inequality \eqref{eq:Ent_Plun-Ruz} as an information-theoretic Pl\"unnecke--Ruzsa-type inequality.

In view of the parallels between \eqref{eq:ruzvol} and \eqref{eq:Ent_Plun-Ruz}, it is natural to ask whether volume is log-submodular, at least up to a constant. Bobkov and the first-named author \cite{BM-12} obtained the first progress in this direction for convex bodies (compact, convex, with nonempty interiors). Let $c_n$ be the smallest possible number satisfying that for any convex bodies $A,B,C$ in $\R^n$, one has
\begin{equation}
    \label{eq:Plun-Ruza_general}
    |A| |A+B+C| \leq c_n |A+B| |A+C|.
\end{equation}
It was proved in \cite{BM-12} that $c_n \leq 3^n$. However, the anticipated log-submodularity property fails for convex bodies, as observed independently by Nayar and Tkocz \cite{NT-17} and Fradelizi, the first and fourth named authors \cite{FMZ-24}. The latter work established the two-sided bounds:
\begin{equation}
    \label{Ineq_Plun-Ruz-sharper}
    \left(\max_{\substack{0\leq i,j\leq n\\k=i+j-n\geq 0}} \frac{\binom{i}{k}\binom{j}{k}}{\binom{n}{k}} \right) \leq c_n \leq \varphi^n,
\end{equation}
where $\varphi = \frac{1+\sqrt{5}}{2}$ is the golden ratio. Using Stirling approximation, one obtains $c_n \geq \frac{2}{\sqrt{\pi n}}\left(\frac{4}{3}\right)^n\left(1+ O(\frac{1}{n})\right)$. While the sharp constant $c_n$ is unknown in dimension $n > 3$, it was shown in \cite{FMZ-24} that $c_3=4/3$. We refer to \eqref{eq:Plun-Ruza_general} as a Pl\"unnecke--Ruzsa-type inequality for convex bodies.
Related studies appear in \cite{FMMZ-24,FHNMZ-26,FHNMZ-26-2,MNZ-24,FMZ-24,N-25}.

Our first result is a Pl\"unnecke--Ruzsa-type inequality for log-concave functions, stated in the following theorem:

\begin{restatable}{theorem}{plunruzzero}
\label{thm: Plu-Ruz-0}
Let $\alpha, \beta \in [-\infty, 0]$, and let $ f,g,h \in C_0 (\R^n)$. Then,
    \begin{equation}\label{eq:Plu-Ruz-0}
        \|f\|_1 \|f \star_\alpha g \star_\alpha h\|_1 \leq c^n \|f \star_\beta g \|_1\|f  \star_\beta h\|_1,
    \end{equation}
    where $c$ is an absolute constant such that $c < 26.$
\end{restatable}
\noindent
Here, $\|f\|_1$ denotes the $L_1$-norm of $f$. The optimal constant remains an open problem. However, it must grow exponentially with the dimension, since by taking $f =1_A, g= 1_B$ and $h =1_C$ for some convex bodies $A,B,C$, we recover \eqref{eq:Plun-Ruza_general}.

We next extend Theorem \ref{thm: Plu-Ruz-0} to the class of $\alpha$-concave functions by comparing $f \in C_\alpha(\R^n)$ with an associated log-concave function, see Corollary \ref{cor:01}.

\begin{restatable}{theorem}{plunruzalpha}
\label{thm: Plu-Ruz-alpha}
Let $\alpha \in (-1/n,0]$, and let $f,g,h \in C_\alpha (\R^n) $. Then,
    \[
        \|f\|_1 \| f \star_\alpha g \star_\alpha h \|_1 \leq \frac{c^n}{\prod_{j=1}^n(1+j \alpha)^2}\| f \star_\alpha g \|_1 \| f \star_\alpha h \|_1,
    \]
    where $c$ is an absolute constant such that $c <26$.
\end{restatable}

We also establish a reverse Pr\'ekopa--Leindler inequality for $\alpha$-concave functions with respect to $\alpha$-sum in Theorem \ref{thm: reverse-PL_alpha}, which recovers the reverse Pr\'ekopa--Leindler inequality for log-concave functions proved by Klartag and V. Milman \cite{KM-05}.

\subsection{Functional liftings of the Rogers--Shephard inequality}

Our next observation relates to the classical Rogers--Shephard inequality \cite{RS-57}, which states that for every convex body $K$ in $\R^n$,
\begin{equation}
    \label{eq:RS-bodies}
    |K-K| \leq \binom{2n}{n} |K|.
\end{equation}
Equality holds if and only if $K$ is a simplex.
Colesanti \cite{C-06} extended this result to quasi-concave functions, proving that for any such function $f$,
\begin{equation}
    \| f \star_{-\infty} \Bar{f}\|_1 \leq \binom{2n}{n} \|f\|_1,
\end{equation}
where $\Bar{f} (x) = f(-x)$ denotes the central reflection of $f$. Equality holds if and only if each superlevel set of  $f$, $\{x : f(x) \geq t \}$, is a simplex for almost every $t$. We note that \cite{C-06} also establishes a Rogers--Shephard-type inequality for $\alpha$-concave functions, but his operation differs from the $\alpha$-sum used in this work. Alonso-Guti\'errez, Gonz\'alez-Merino, Jim\'enez and Villa~\cite{AGJV-16} proved, when $n\geq 2$, an inequality for log-concave functions with full-dimensional support under the Asplund product and showed that equality holds if and only if  $f$ is the characteristic function of a simplex.
The equality cases in dimension one differ substantially from those in higher dimensions.

\begin{restatable}{theorem}{RSAlphaSum}\label{RS-QC-inte}
    Let $\alpha \in (-\infty,0]$, and let $ f \in C_\alpha(\R^n) $ be an integrable function with full-dimensional support. Then,
    \begin{equation} \label{RS-QC-eq-intro}
	\|f \star_{\alpha} \Bar{f}\|_1 \leq \binom{2n}{n}\|f\|_1.
    \end{equation}
    For $n\geq 2$, equality holds if and only if $f$ is a characteristic function of a simplex. For $n=1$, equality holds if and only if $f $ is monotone on its support.
\end{restatable}
We refer to \cite{A-19,AAGJV-19,AGJV-16,AHMRZ-21,C-06,FMMN-26-1,FMMN-26-2,KM-26} for the study of the Rogers--Shephard-type inequality and its functional analogues.

\subsection{Functional lifting of sum-difference inequalities}

We recall Litvak's observation (see \cite[pp. 534]{S-14} and also in \cite{FMZ-24}): for any convex bodies $A,B$, one has
\begin{equation} \label{eq:Litvak Ob}
    |A+B| \leq \frac{1}{2^n}\binom{2 n}{n} |A - B|.
\end{equation}
with equality when $A= - B$ and $A$ is a simplex. We prove some functional analogues of \eqref{eq:Litvak Ob}. In particular, we obtain sharp analogs for the $\alpha$-convolution when $\alpha$ is $-\infty$ or $0$.
\begin{restatable}{prop}{LitvakFuncQuasi}\label{p:litvak-fun}
    Let $f,g \in C_{-\infty}(\R^n)$ be integrable functions. Then
    \begin{equation}
        \| f \star_{-\infty} g\|_1 \leq \frac{1}{2^n}\binom{2 n}{n} \| f \star_{-\infty} \Bar{g}\|_1.
    \end{equation}
    Equality holds if  $f= \bar{g}$ and its superlevel sets $U_t(f)$ are simplices for almost every $t>0$.
\end{restatable}

The simple proof is given in Section~\ref{sec:sumdiff}.

\begin{prop}
    \label{prop:litvak-fun-0}
    Let $f,g \in C_{0}(\R^n)$. Then
    \begin{equation} \label{eq:Litvak-alpha-0cw}
        \| f \star_0 g\|_1 \leq 2^n\| f \star_0 \Bar{g}\|_1.
    \end{equation}
    Moreover, the constant is sharp.
\end{prop}

Theorem~\ref{prop:litvak-fun} generalizes this to $\alpha\in (-1/n,0)$, but at the cost of possibly non-sharp constants.

\subsection{Functional lifting of the Ruzsa triangle inequality}

Finally, we investigate inequalities inspired by Ruzsa's triangle inequality \cite{R-96}. As a by-product, we obtain an improvement of the known upper bound for the constant in the Rogers--Shephard-type inequality with Colesanti’s operation in Proposition \ref{eq:improved-diff-alpha}.

The Euclidean form of Ruzsa's triangle inequality asserts that for any compact sets \(A,B,C\subset \mathbb{R}^n\),
\begin{equation}
\label{eq:Ruzsa-ineq-wo-A}
|A|\,|B-C| \le |A-B|\,|A-C|.
\end{equation}

We present a functional analogue of Ruzsa's inequality.

\begin{theorem} \label{thm: Ruza-triangle-fun}
    Let $f,g,h \in C_0 (\R^n)$. Then
    \[
        \|f\|_1 \|g \star_0 \Bar h\|_1 \leq 2^n \|f \star_0 \Bar g \|_1 \|f \star_0 \Bar h \|_1.
    \]
    Moreover, the constant is sharp.
\end{theorem}

Theorem~\ref{thm:ruza-functional} extends this inequality to $\alpha\in (-\frac{1}{n},0)$, but with possibly non-sharp constants. We also establish a functional version of \eqref{eq:ruzvol1} in Corollary \ref{cor:ruzvol1-func} as a consequence of Theorem \ref{thm: Ruza-triangle-fun} and Theorem~\ref{prop:litvak-fun}.

\par\vspace{.2in}
\noindent{\bf Techniques.}
Our main approach is the development of sharp inequalities for the base transform associated with $\alpha$-concave functions, which we combine with the moment-comparison method (as used, for example, by Fradelizi, Li, and the first-named author \cite{FLM-20}). The latter method reduces integral estimates for $\alpha$-concave functions to comparisons between normalized one-dimensional moments of their level-set profiles, and the base inequalities allow us to apply this framework to functional convolution estimates.

\par\vspace{.2in}
\noindent{\bf Outline of the paper.}
The paper is organized as follows. Section~\ref{previous-results} contains notations and known results. Section~\ref{our-main-results} is devoted to establishing Theorem \ref{thm: Plu-Ruz-0} and \ref{thm: Plu-Ruz-alpha}, the Pl\"unnecke--Ruzsa-type inequalities, and also the reverse Pr\'ekopa--Leindler-type inequality in Theorem~\ref{thm: reverse-PL_alpha}. In Section~\ref{sec:RS}, we prove Theorem \ref{RS-QC-inte} with the equality characterization for all $n$. Section~\ref{sec:sumdiff} contains the functional inequality of sum-difference inequalities, Theorem \ref{prop:litvak-fun}, Proposition \ref{eq:improved-diff-alpha}, and Proposition \ref{p:litvak-fun}. Section~\ref{sec:triangle} is dedicated to the functional inequality of Ruzsa triangle inequality, Theorem \ref{thm: Ruza-triangle-fun}. Finally, the Appendix \ref{appendixtech} collects the approximations used throughout the paper.

\section{Notations and known results}\label{previous-results}

We refer to \cite{S-14} for standard definitions and facts used. For any $x \in \R^n$, $|x|_p$ denotes its $\ell_p$-norm. A convex body is defined to be a convex, compact set with a nonempty interior.
We write $ |K| $ for the $n$-dimensional Lebesgue measure (volume) of a measurable set $K\subset \R^n$. For any sets $K$ and $L$, the Minkowski sum of $K $ and $ L$ is defined as $K+L = \{ x+y : x\in K, y \in L \}$.

\subsection{\texorpdfstring{$\alpha$}{α}-concave functions} \label{section: alpha-concave}

We repeat the definition of $\alpha$-concavity for convenience.
Fix \( \alpha \in [-\infty, \infty] \). We say that a function \( f : \mathbb{R}^n \rightarrow [0, \infty ) \) is \( \alpha \)-concave if \( f \) is supported on a convex set \( \Omega \subset \mathbb{R}^n \), and for all \( x, y \in \Omega \) and \( \lambda \in [0,1] \), we have
\begin{equation} \label{alp-conc}
    f(\lambda x + (1 - \lambda)y)
    \geq \left[\lambda f(x)^\alpha + (1 - \lambda)f(y)^\alpha\right]^{1/\alpha}.
\end{equation}
For $ \alpha\in\{-\infty, 0, \infty\}$, the above inequality is understood in the limiting sense. Thus, when $\alpha = 0$, the condition becomes
\begin{equation}
    f(\lambda x + (1 - \lambda)y)
    \geq f(x)^\lambda f(y)^{1 - \lambda},
\end{equation}
and when $\alpha = -\infty$, it becomes
\begin{equation}
    f(\lambda x + (1 - \lambda)y)
    \geq \min \left\{ f(x) ,  f(y) \right\}.
\end{equation}
We recall Definition \ref{def:C_alpha}.
\begin{defn}
For $-\infty \leq \alpha \leq 0$, we write
\begin{equation*}
C_\alpha(\mathbb{R}^n)=
\left\{f: \mathbb{R}^n\to \mathbb{R}_+  \Bigg|\begin{aligned}
\,\,\, & f \text{ is upper semicontinuous and $\alpha$-concave}  , \\
&\|f\|_{\infty}=f(0)=1, \lim_{|x| \to \infty}  f(x) = 0
\end{aligned}\right\}.
\end{equation*}
\end{defn}
If \( \alpha_1 < \alpha_2 \), then \( C_{\alpha_1}(\mathbb{R}^n) \supset C_{\alpha_2}(\mathbb{R}^n) \). Thus, the class \( C_\alpha(\mathbb{R}^n) \) can be viewed as an extension of the class \( C_0(\mathbb{R}^n) \) of log-concave functions.
In some papers, these functions are referred to as  $s$-concave functions. The discussion and many results on $ \alpha$-concave functions can be found in \cite{R-13,MR-12, MR-13,B-10,R-14}.

The natural correspondence between log-concave functions and convex functions is as follows: for any log-concave function \( f \),
the function \( \varphi = -\log f \) is convex.
V.~Milman and Rotem~\cite{MR-13} extended this correspondence to \( \alpha \)-concave functions: Fix $ \alpha \in (-\infty, 0) $. The convex base of \( f \) is
\[
    \operatorname{base}_\alpha f = \frac{1 - f^\alpha}{\alpha},
\]
where the value  is $+ \infty$ at points at which $ f = 0$. For $\alpha = 0$, we set $\operatorname{base}_0(f) = - \log f$.
With our assumption, $\operatorname{base}_\alpha(f)$ is a convex, lower semicontinuous function with minimum 0 at the origin, and $\operatorname{base}_\alpha(f) \rightarrow \infty$ as $|x| \rightarrow \infty$.
We recall the infimal convolution of convex functions $\varphi, \psi$,
\[
    (\varphi \mathbin\square \psi)(x):=\inf _{y+z=x}\{\varphi(y)+\psi(z)\}.
\]
The \( \alpha \)-sum of $ f $ and $ g $ is defined by
\begin{equation}
    \label{def:alpha-sum-original}
    (\operatorname{base}_\alpha\left(f \star_\alpha g\right)) (x) :=  (\operatorname{base}_\alpha f \mathbin\square \operatorname{base}_\alpha g) (x) .
\end{equation}
Note that for the characteristic function of convex sets $K$ and $L$, we have $1_K \star_\alpha 1_L = 1_{K+L}$ for all $\alpha$.
Moreover, for any \( -\infty < \alpha < 0 \), the \( \alpha \)-sum can be written explicitly, without referring to convex base functions:
\begin{equation}
\label{eq:explicit-form-some-alpha-sum}
    (f \star_\alpha g) (x) = \sup_{ y +z =x } (f(y)^\alpha + g(z)^\alpha -1)^{1/\alpha},
\end{equation}
where the expression is understood to be zero whenever $f(y)=0$ or $g(z)=0$.
For $\alpha=0 $ and $\alpha = -\infty$, the operation is defined by the corresponding limit. In particular, the $0 $-sum is the Asplund product (or sup-convolution),
\begin{equation}
    \label{def:asplund-sum}
    (f \star_0 g) (x) = \sup_{y+z = x} f(y) g(z).
\end{equation}
The $(-\infty)$-sum, also called the quasi-sum, is
\[
    (f \star_{-\infty} g) (x) = \sup_{y+z = x} \min \{ f(y), g(z) \}.
\]
Notice that $C_\alpha(\mathbb{R}^n)$ is closed under the $\alpha$-sum as well as under the quasi-sum.
Since the quasi-sum plays a distinguished role in the theory of mixed integrals for quasi-concave functions \cite{MR-12}, it is typically denoted by a special notation $\oplus$ instead of $\star_{-\infty}$. The following level-set identity is a key tool to establish the mixed integral. Recall that for a function $f$ and $t\in[0,1]$, we define the superlevel set
$
    U_t(f):=\{x\in\mathbb{R}^n:\ f(x)\ge t\}.
$
It was proved in \cite[Theorem~8]{MR-12} that for any $f,g\in C_{-\infty}(\mathbb{R}^n)$ and any $t>0$,
\begin{equation} \label{sum-Uf}
    U_t(f \oplus g) = U_t(f) +U_t(g).
\end{equation}
Under our assumption, $U_t(f)$ is convex and compact for every $t >0$ if and only if $f \in C_{-\infty} (\R^n)$.
The following identity is known as the layer-cake representation:
\begin{equation} \label{mix-integral-for}
    \|f\|_1 = \int_0^1 |U_t(f)| \;\mathrm d t.
\end{equation}

\section{Pl\"unnecke--Ruzsa-type inequality for functions} \label{our-main-results}

\subsection{Preliminaries}
We begin with an auxiliary lemma that will be used in this section as well as in Section~\ref{sec:RS}, and it is known as Karamata's inequality \cite{K-32}.

\begin{lemma}
    \label{l:Karamata}
    Let $x,y \geq 1$ and let $\varphi$ be convex on $[0,\infty)$. Then
    \begin{equation}
        \label{eq:Karamata}
        \varphi(x) +\varphi(y) \leq \varphi(1) + \varphi(x+y-1).
    \end{equation}
\end{lemma}

\begin{proof}
    Assume that $x,y \neq 1$; otherwise, the inequality is trivial. Using the convexity with $\lambda = \frac{y- 1}{x+y-2} \in (0,1)$, we have
    \[
        \varphi(x) = \varphi (\lambda  +(1-\lambda)(x+y-1)) \leq \lambda \varphi(1) + (1-\lambda) \varphi(x+y -1).
    \]
    Similarly, we obtain
    \[
        \varphi(y) \leq (1-\lambda) \varphi(1) + \lambda \varphi(x+y -1).
    \]
    Adding the two inequalities proves the desired inequality.
\end{proof}
Using \eqref{eq:explicit-form-some-alpha-sum} and Lemma~\ref{l:Karamata} with the function $t \mapsto t^s$ for any $s >1$, we obtain the following monotonicity property of the $\alpha$-sum of $f$ and $g$.
\begin{prop} \label{prop:monotonicity}
    For any $f,g \in C_{\beta}(\R^n)$ and for any $-\infty\leq\alpha<\beta\leq 0$, we have
    \begin{equation} \label{monotonicity}
        f \star_\alpha g \geq f \star_\beta g.
    \end{equation}
\end{prop}

\begin{proof}
    Fix \(x\in\mathbb R^n\). By the definition, \(f\) and \(g\) take values in \([0,1]\). It is enough to consider \(y,z\in\mathbb R^n\) with \(y+z=x\) and \(f(y),g(z)>0\).
    We prove the desired inequality pointwise for each such pair $(y, z)$, and then take the supremum over $y,z$ such that $x=y+z$.
\begin{enumerate}[wide =0pt, labelwidth= .9cm, leftmargin=1.55cm, label=\textbf{Case \arabic*}:]
        \item Suppose that \( -\infty<\alpha<\beta<0.\)  Since $ \frac{\alpha}{\beta} \geq 1 $, the function $t \mapsto t^{\alpha / \beta}$ is convex.
    Applying Lemma \ref{l:Karamata}, we obtain
    \[
        f(y)^\alpha + g(z)^\alpha \leq 1 + (f(y)^\beta + g(z)^\beta - 1)^{\alpha / \beta}.
    \]
    Subtracting $1$ from both sides, and then raising both sides to the power of ${1} / {\alpha} <0$, we get
    \[
        \left(f(y)^\alpha + g(z)^\alpha - 1\right)^{1/\alpha} \geq (f(y)^\beta + g(z)^\beta - 1)^{1 / \beta}.
    \]
    \item Assume that $\beta =0$ and $-\infty <\alpha<0$. Since
    \[
        0 \leq (f(y)^\alpha -1)(g(z)^\alpha -  1 ) = f(y)^\alpha g(z)^\alpha - f(y)^\alpha - g(z)^\alpha + 1 ,
    \]
    we obtain
    \[
        f(y)^\alpha + g(z)^\alpha - 1 \leq f(y)^\alpha g(z)^\alpha.
    \]
    Raising to the power of ${1} / {\alpha} <0$ gives the pointwise inequality.
    \item Suppose that $\beta <0,$ and $\alpha = -\infty $. Without loss of generality, assume that $f(y) \leq g(z) .$ Thus,
    \[
        (\min \{ f(y), g(z)\})^\beta = f(y)^\beta \leq  f(y)^\beta+g(z)^\beta-1 .
    \]
    Raising to the power of ${1} / {\beta} <0$, we obtain the pointwise inequality.
    \item Finally, if $\alpha=-\infty$ and $\beta=0$, then, because $f(y), g(z) \in[0,1]$,
    \[
        \min \{f(y), g(z)\} \geq f(y) g(z).
    \]
    \end{enumerate}
\end{proof}

\begin{lemma} \label{lem:monotonicity-V}
    Let $\varphi: \mathbb{R}^n \rightarrow[0, \infty]$ be convex with $\varphi(0)=0$.  Then, the function
    \[
        t \mapsto \frac{\left|\{x: \varphi(x) \leq t\}\right|}{t^n}
    \]
    is non-increasing.
\end{lemma}

\begin{proof}
    Fix $0< s<t$. Let $x \in \{ x: \varphi (x)\leq t\}$. Using the convexity of $\varphi$, we obtain
    $$
        \varphi\left(\frac{s}{t} x\right) \leq \frac{s}{t} \varphi(x)+\left(1-\frac{s}{t}\right) \varphi(0) \leq s.
    $$
    Thus,
    $$
        \frac{s}{t} \{ x: \varphi (x)\leq t\} \subset \{ x: \varphi (x)\leq s\}.
    $$
    Taking volumes yields
$$
\frac{|\{x: \varphi(x) \leq t\}|}{t^n} \leq \frac{|\{x: \varphi(x) \leq s\}|}{s^n} .
$$
Hence, the function is non-increasing on $( 0, \infty )$.
\end{proof}

\begin{lemma} \label{lem:sublevel-set-representation}
    Let $\varphi: \mathbb{R}^n \rightarrow[0, \infty]$ be measurable. Let $f:[0, \infty) \rightarrow[0, \infty)$ be locally absolutely continuous and non-increasing, with
    $
    \displaystyle \lim_{t \rightarrow \infty} f(t)=0 .
    $
    Then,
    $$
    \int_{\mathbb{R}^n} f(\varphi(x)) \;\mathrm d x=\int_0^{\infty}|\{x: \varphi(x) \leq t\}|\left(-f^{\prime}(t)\right) \;\mathrm d t,
    $$
    where the equality is understood in $[0, \infty]$.
\end{lemma}
\begin{proof}
    Since $\displaystyle \lim_{t \rightarrow \infty} f(t)=0 ,$ we have
    $$
    f(u)=\int_u^{\infty}-f^{\prime}(t) \;\mathrm d t, \qquad u \geq 0.
    $$
    The integrand is non-negative because $f$ is non-increasing. Applying Tonelli's theorem, we obtain
    \begin{align}
        \int_{\mathbb{R}^n} f(\varphi(x)) \;\mathrm d x
         &=\int_{\mathbb{R}^n} \int_{\varphi(x)}^{\infty} -f^{\prime}(t)\;\mathrm d t \;\mathrm d x
         =\int_0^{\infty} \int_{\{\varphi \leq t\}} -f^{\prime}(t) \;\mathrm d x \;\mathrm d t
         \\
         &=\int_0^{\infty}|\{x: \varphi(x) \leq t\}|\left(-f^{\prime}(t)\right) \;\mathrm d t .
    \end{align}
\end{proof}
For probability measures $\mu$ and $\nu$ on $\mathbb R$, we say that $\mu$ is smaller than $\nu$ in the stochastic order, and write $\mu \preceq_{\text {st }} \nu$
if
\[
\mu((t, \infty)) \leq \nu((t, \infty)) \quad \text { for every } t \in \mathbb{R} .
\]
The following lemmas are standard consequences of the characterization of the usual stochastic order \cite[Section 1.A.1]{SS-07}.
\begin{lemma} \label{lem:measures-comparison}
    Let $\mu_1$ and $\mu_2$ be probability measures on $(0, \infty)$ with densities $p_1$ and $p_2$, respectively. Suppose there exists $c>0$ such that
    $$
    p_1(t) \leq p_2(t) \quad \text { for a.e. } t \in(0, c),
    \quad \text{and} \quad
    p_1(t) \geq p_2(t) \quad \text { for a.e. } t \in(c, \infty) .
    $$
    Then, $\mu_2 \preceq_{st} \mu_1$.
\end{lemma}

\begin{lemma} \label{lem:measures-comparison-non-increasing}
    Let $\mu_1$ and $\mu_2$ be probability measures on $(0,\infty)$ with $\mu_2 \preceq_{st} \mu_1$. Then, for every measurable, non-increasing function $V:(0, \infty) \rightarrow[0, \infty]$, we have
    \begin{equation}
    \int_{(0, \infty)} V \;\mathrm d \mu_1 \leq \int_{(0, \infty)} V \;\mathrm d \mu_2,
    \end{equation}
    where the integrals are understood in $[0, \infty]$.
\end{lemma}

\subsection{Log-concave functions}
        Recall the lower incomplete gamma function
        \[
            \gamma(t) = \gamma(n+1, t) = \int_0^t u^{n} e^{-u} \;\mathrm d u.
        \]
\begin{lemma} \label{lem:stoc-1}
    Let $\mu_1$, $\mu_2$ be measures with the following densities, respectively,
    \[
        p_1 (t) = \frac{t^{2n}e^{-t}(t +e^{t / 2}-1)}{2(4^n+n)(2n)!}, \quad p_2 (t) = \frac{2 \gamma(t) \gamma '(t)}{(n!)^2}.
    \]
    Then, $\mu_1$ and $\mu_2$ are probability measures on $(0,\infty)$, and $\mu_2  \preceq_{st} \mu_1$.
\end{lemma}
\begin{proof}
    Since
        $$
            \int_0^\infty t^{2n}e^{-t}(t +e^{t / 2}-1)\;\mathrm  d t
        $$
        $\mu_1$ is a probability measure.
        Also, $\mu_2$ is a probability measure since
        \[
            \int_0^\infty p_2(t) \;\mathrm dt= \frac{\gamma^2(t)}{(n!)^2} \Bigg|_0^\infty = \frac{\Gamma(n+1)^2}{ (n!)^2}  =1.
        \]
        By Lemma \ref{lem:measures-comparison}, it is enough to prove that there exists $c>0$ such that
        $$
        p_1(t) \leq p_2(t) \quad \text { for a.e. } t \in(0, c),
        \quad \text{and} \quad
        p_1(t) \geq p_2(t) \quad \text { for a.e. } t \in(c, \infty) .
        $$
        Define the ratio
        \[
            R (t) = \frac{p_1(t)}{p_2(t)}, \qquad t >0
            .
        \]
        Using the fact that $ \gamma ' (t) = t^ne^{-t}$, we obtain
        \begin{align}
            R (t) &= \frac{t^{2n}e^{-t}(t +e^{t / 2}-1)}{2(4^n+n)(2n)!} \cdot \frac{(n!)^2}{2 \gamma(t) \gamma '(t)}
            = \frac{(n!)^2}{4(4^n+n)(2n)!} \cdot \frac{t^{n}(t +e^{t / 2}-1)}{\gamma(t) }
            .
        \end{align}
        Note that $R$ is positive since $e^{t/2} \geq 1$. Its logarithmic derivative is
        $$
        \frac{d}{d t} \log R(t)=\frac{n}{t} +\frac{1+\frac{1}{2} e^{t / 2}}{t+e^{t / 2}-1} -\frac{\gamma'(t)}{\gamma(t)}.
        $$
        Observe that
        \[
            e^{t / 2}-1=\int_0^{t / 2} e^u \;\mathrm  d u \leq \int_0^{t / 2} e^{t/2} \;\mathrm d u = \frac{t}{2}e^{t/2}.
        \]
        Hence,
        $$
            \frac{1+\frac{1}{2} e^{t / 2}}{t+e^{t / 2}-1} \geq \frac{1}{t}.
        $$
        Also,
        \[
            \gamma(t) = \int_0^t u^n e^{-u}  \;\mathrm d u \geq \int_0^t u^n e^{-t}  \;\mathrm d u \geq e^{-t} \frac{t^{n+1}}{n+1}.
        \]
        Since $ \gamma ' (t) = t^ne^{-t}$,
        $$
            \frac{\gamma'(t)}{\gamma(t)} \leq \frac{n+1}{t}.
        $$
        Therefore,
        \[
            \frac{d}{d t} \log R(t) \geq \frac{n}{t} +\frac{1}{t} - \frac{n+1}{t} = 0.
        \]
        Thus, $R$ is non-decreasing.
        Using the fact that $p_2(t)>0$ and that $p_1$ and $p_2$ have the same total mass, we have
        $$
        0=\int_0^{\infty}p_1(t)-p_2(t) \;\mathrm  d t=\int_0^{\infty}(R(t)-1) p_2(t) \;\mathrm  d t.
        $$
        Consequently, $R-1$ cannot be everywhere positive or everywhere negative. Therefore, by the continuity of $R$, there exists $c>0$ such that
        $$
        R(t) \leq 1 \quad \text { for a.e. } t \in(0, c),
        \quad \text{and} \quad R(t) \geq 1 \quad \text { for a.e. } t \in(c, \infty) ,
        $$
        which is equivalent to the desired comparison between $p_1$ and $p_2$.
\end{proof}
\plunruzzero*
\begin{proof}
    We prove the stronger inequality
    \begin{equation} \label{eq:desired-ineq-thm1}
        \|f\|_1 \|f \oplus g \oplus h\|_1 \leq c_n (4^n +n) \binom{2n}{n} \|f \star_0 g \|_1\|f  \star_0 h\|_1,
    \end{equation}
    where $c_n$ is the best possible constant for volume inequality \eqref{eq:Plun-Ruza_general}.
    Then, Theorem \ref{thm: Plu-Ruz-0} follows from Proposition \ref{prop:monotonicity}.
    The estimate for this constant is given in Lemma \ref{upper-bound-c}.
    It follows from the level set identity \eqref{sum-Uf} that
    $$
        U_t (f \oplus g \oplus h) = U_t (f) + U_t (g)  + U_t (h).
    $$
    Using the layer-cake representation \eqref{mix-integral-for}, we get
    \begin{align}
        \|f\|_1 \|f \oplus g \oplus h\|_1  &= \int_0^1 |U_s (f)  | \;\mathrm ds\int_0^1 |U_t (f \oplus g \oplus h) | \;\mathrm dt
        \\
        &=
        \int_0^1 \int_0^1 |U_s (f)  | |U_t (f) + U_t (g)  + U_t (h) | \;\mathrm ds \;\mathrm dt.
    \end{align}
    We split the domain of integration into the two regions: $s <t$ and $s\geq t$.
        Suppose that $0<s< t<1$. Then, $U_t(f) \subset U_s(f)$. Using the volume inequality \eqref{eq:Plun-Ruza_general}, we obtain
        \begin{align}
            |U_s (f)  | |U_t (f) + U_t (g)  + U_t (h) |
            &\leq |U_s (f)  | |U_s (f) + U_t (g)  + U_t (h) |
            \\
            &\leq c_n |U_s (f) + U_t (g) | |U_s (f)   + U_t (h) |.
        \end{align}
        Observe that
        $$
        U_s(f)+U_t(g) \subseteq U_{s t}(f \star_0 g), \quad U_s(f)+U_t(h) \subseteq U_{s t}\left(f \star_0 h\right) .
        $$
        Indeed, if $y \in U_s(f)$ and $z \in U_t(g)$, then $f(y) \geq s$ and $g(z) \geq t$. Hence,
        $$
        \left(f \star_0 g\right)(y+z) \geq f(y) g(z) \geq s t .
        $$
        Therefore, $y+z \in U_{s t}\left(f \star_0 g\right)$, proving the inclusion. The second inclusion follows by the same argument, replacing $g$ by $h$.
        Hence,
        \[
            |U_s (f)  | |U_t (f) + U_t (g)  + U_t (h) |  \leq c_n |U_{s t}(f \star_0 g) | | U_{s t}(f \star_0 h)|.
        \]
        Thus, the contribution from the domain $\{s \leq t\}$ is
        \[
            \iint_{\{s< t\}} |U_s (f)  | |U_t (f) + U_t (g)  + U_t (h) | \;\mathrm ds \;\mathrm dt \leq c_n \iint_{\{s< t\}}  |U_{s t}(f \star_0 g) | | U_{s t}(f \star_0 h)| \;\mathrm ds \;\mathrm dt.
        \]
        Using the change of variable $r = st$ and then Fubini's theorem, we obtain
        \begin{align}
        &\iint_{\{s\leq t\}} |U_s (f)  | |U_t (f) + U_t (g)  + U_t (h) | \;\mathrm ds \;\mathrm dt
        \\
        &\leq c_n \int_0^1 \int_{0}^{t^2} |U_{r}(f \star_0 g) | | U_{r}(f \star_0 h)| \;\frac{\mathrm dr}{t} \;\mathrm dt
        = c_n  \int_{0}^{1} \int_{\sqrt{r}}^1 |U_{r}(f \star_0 g) | | U_{r}(f \star_0 h)|  \;\frac{\mathrm dt}{t} \;\mathrm dr
        \\
        &=\frac{c_n}{2} \int_0^1 \log \left(\frac{1}{r}\right)|U_{r}(f \star_0 g) | | U_{r}(f \star_0 h)|  \;\mathrm dr. \label{eq:first-con}
        \end{align}
        Now, we estimate the contribution from the domain $0<t \leq s <1$. Using a similar argument, we obtain
        \[
            |U_s (f)  | |U_t (f) + U_t (g)  + U_t (h) |  \leq c_n |U_{t^2}(f \star_0 g) | | U_{t^2}(f \star_0 h)|.
        \]
        Thus, using the change of variable $r = t^2$,
        \begin{align}
            &\iint_{\{t < s\}} |U_s (f)  | |U_t (f) + U_t (g)  + U_t (h) | \;\mathrm ds \;\mathrm dt
            \\
            &\leq c_n \iint_{\{ t<s\}}  |U_{t^2}(f \star_0 g) | | U_{t^2}(f \star_0 h)| \;\mathrm ds \;\mathrm dt
            =  c_n \int_{0}^1 (1-t) |U_{t^2}(f \star_0 g) | | U_{t^2}(f \star_0 h)|  \;\mathrm dt
            \\
            &= \frac{c_n}{2} \int_0^1\left(r^{-1 / 2}-1\right) |U_{r}(f \star_0 g) | | U_{r}(f \star_0 h)|  \;\mathrm dr. \label{eq:second-con}
        \end{align}
        Adding the two contributions, we obtain
        \begin{align}
            \|f\|_1 \|f \oplus g \oplus h\|_1
            & \leq \frac{c_n}{2} \int_0^1\left(\log \left(\frac{1}{r}\right)+r^{-1 / 2}-1\right)\left|U_r\left(f \star_0 g\right)\right|\left|U_r\left(f \star_0 h\right)\right|\;\mathrm  d r.
        \end{align}
        Denote
        \[
            \varphi_1 = \operatorname{base}_0 (f \star_0 g) = -\log ( f \star_0 g), \quad \text{and} \quad \varphi_2 = \operatorname{base}_0 (f \star_0 h)= -\log ( f \star_0 h).
        \]
        Thus, $U_r(f \star_0 g) = \{x : \varphi_1 (x) \leq -\log r \}$ and $U_r(f \star_0 h) = \{x : \varphi_2 (x) \leq -\log r \}.$ Using the change of variables $ r = e^{-t}$, we have
        \begin{align}
            \|f\|_1 \|f \oplus g \oplus h\|_1
            & \leq \frac{c_n}{2} \int_0^1\left(\log \left(\frac{1}{r}\right)+r^{-1 / 2}-1\right)\left|U_r\left(f \star_0 g\right)\right|\left|U_r\left(f \star_0 h\right)\right|\;\mathrm  d r
            \\
            & = \frac{c_n}{2} \int_0^\infty e^{-t}(t +e^{t / 2}-1)\left|\{x : \varphi_1 (x) \leq t\}\right|\left|\{x : \varphi_2 (x) \leq t \}\right|\;\mathrm  d t. \label{eq:left-desired-0}
        \end{align}
        Denote
        \[
            V_1 (t) = \frac{|\{x : \varphi_1 (x) \leq t \}|}{t^n}, \quad \text{and} \quad  V_2 (t) = \frac{|\{x : \varphi_2 (x) \leq t \}|}{t^n}.
        \]
        Using Lemma \ref{lem:monotonicity-V}, $V_1$ and $V_2$ are non-increasing. Since $V_1,V_2 \geq 0$, their product $V_1V_2$ is also non-increasing.
        Now, we define measure $\mu_1,\mu_2 $ on $(0,\infty)$ with densities
        \[
            p_1 (t) = \frac{t^{2n}e^{-t}(t +e^{t / 2}-1)}{2(4^n+n)(2n)!},
            \qquad
            p_2 (t) = \frac{2 \gamma(t) \gamma '(t)}{(n!)^2}.
        \]
        By Lemma \ref{lem:stoc-1}, $\mu_1,\mu_2$ are probability measures. Moreover, there exists $c>0$ such that
        $$
        p_1(t) \leq p_2(t) \quad \text { for a.e. } t \in(0, c),
        \quad \text{and} \quad
        p_1(t) \geq p_2(t) \quad \text { for a.e. } t \in(c, \infty) .
        $$
        By Lemma \ref{lem:measures-comparison}, we have $\mu_2 \preceq_{st} \mu_1$.
        Therefore, \eqref{eq:left-desired-0} can be rewritten as
        \begin{align}
            \|f\|_1 \|f \oplus g \oplus h\|_1 \leq
            (4^n+n)(2n)!c_n \int_{0}^{\infty} V_1 (t) V_2 (t) \;\mathrm d \mu_1 (t). \label{eq:left-desired-0-ver2}
        \end{align}
        Applying Lemma \ref{lem:measures-comparison-non-increasing} to \eqref{eq:left-desired-0-ver2} and the fact that $ \;\mathrm d\mu_2 = \frac{2 \gamma(t) \gamma '(t)}{(n!)^2} \;\mathrm d t$, we obtain
        \begin{align}
            \|f\|_1 \|f \oplus g \oplus h\|_1
            &\leq
            (4^n+n)(2n)! c_n \int_{0}^{\infty} V_1 (t) V_2 (t) \;\mathrm d \mu_2 (t)
            \\
            &= 2(4^n+n)\binom{2n}{n}c_n \int_{0}^{\infty} V_1 (t) V_2 (t) \gamma(t) \gamma '(t) \;\mathrm d t. \label{eq:left-desired-0-ver3}
        \end{align}
        Observe that
        \[
            \int_{0}^{\infty}\int_{0}^{\infty}\gamma'(s) \gamma'(t) V_1(\max\{s,t\}) V_2(\max\{s,t\})  \;\mathrm d s \;\mathrm d t = 2\int_{0}^{\infty} V_1 (t) V_2 (t) \gamma(t) \gamma '(t) \;\mathrm d t.
        \]
        Indeed, split the domain of integration into the two regions $s \leq t$ and $t < s$. On the first region, we obtain
        \begin{align}
            \iint_{\{s\leq t\}} \gamma'(s) \gamma'(t) V_1(\max\{s,t\}) V_2(\max\{s,t\})  \;\mathrm d s \;\mathrm d t
            &= \int_0^\infty \int_{0}^t \gamma'(s) \gamma'(t) V_1(t) V_2(t)  \;\mathrm d s \;\mathrm d t
            \\
            &= \int_{0}^{\infty} V_1 (t) V_2 (t) \gamma(t) \gamma '(t) \;\mathrm d t.
        \end{align}
        By symmetry, the integral over the region $\{t<s\}$ has the same value. Adding them gives the claimed identity. Substituting this identity into \eqref{eq:left-desired-0-ver3}, then using the fact that $V_1,V_2$ are non-increasing, and finally applying Fubini's theorem, we obtain
        \begin{align}
            \|f\|_1 \|f \oplus g \oplus h\|_1
            &\leq  (4^n+n)\binom{2n}{n}c_n \int_{0}^{\infty}\gamma'(s) \gamma'(t) V_1(\max\{s,t\}) V_2(\max\{s,t\})  \;\mathrm d s \;\mathrm d t
            \\
            &\leq (4^n+n)\binom{2n}{n}c_n \int_{0}^{\infty} \gamma'(s) \gamma'(t) V_1(s) V_2(t)  \;\mathrm d s \;\mathrm d t
            \\
            &=
            (4^n+n)\binom{2n}{n}c_n  \int_0^{\infty}V_1 (s) \gamma '(s)\;\mathrm d s\int_0^{\infty}V_2 (t) \gamma '(t)\;\mathrm d t.
        \end{align}
        Using Lemma \ref{lem:sublevel-set-representation} with the function $ t \mapsto e^{-t}$ and the fact that $\gamma '(t)=t^ne^{-t}$, we get
        \begin{equation}
        \int_{\mathbb{R}^n} (f \star_0 g)(x) \;\mathrm d x=\int_0^{\infty}|\{x: \varphi_1(x) \leq t\}|e^{-t} \;\mathrm d t =\int_0^{\infty}t^ne^{-t}V_1 (t) \;\mathrm d t =\int_0^{\infty}V_1 (t) \gamma '(t)\;\mathrm d t, \label{eq:left-desired-1}
        \end{equation}
        and
        \begin{equation}
            \int_{\mathbb{R}^n} (f \star_0 h)(x) \;\mathrm d x=\int_0^{\infty}|\{x: \varphi_2(x) \leq t\}|e^{-t} \;\mathrm d t =\int_0^{\infty}t^ne^{-t}V_2 (t) \;\mathrm d t=\int_0^{\infty}V_2 (t) \gamma '(t)\;\mathrm d t. \label{eq:left-desired-2}
        \end{equation}
        Substituting these identities in the above inequality completes the proof.
\end{proof}

\begin{remark}
    As noted in the introduction, the best constant $c$ must be greater than 1, and in fact, since \eqref{eq:Plu-Ruz-0} specializes to the corresponding inequality for volumes of convex bodies by considering indicator functions, we must have $c\ge 4/3 -o(1)$ for large $n$.
\end{remark}

\begin{remark}
    When $\alpha = 0 = \beta$, Theorem \ref{thm: Plu-Ruz-0} remains valid for arbitrary nontrivial integrable log-concave functions. Let $x_0,y_0,z_0$ be maximizers of $f,g$, and $h$ respectively.
    After translations and normalizations, Theorem \ref{thm: Plu-Ruz-0} gives
    \[
        \frac{\|f\|_1 \|f\star_0 g \star_0 h\|_1}{\|f\star_0 g \|_1 \|f\star_0  h\|_1} =
        \frac{\|F\|_1 \|F\star_0 G \star_0 H\|_1}{\|F\star_0 G \|_1 \|F\star_0  H\|_1} \leq c^n
    \]
    where
    $$
    F(x)=\frac{f\left(x+x_0\right)}{f\left(x_0\right)}, \quad G(x)=\frac{g\left(x+y_0\right)}{g\left(y_0\right)}, \quad H(x)=\frac{h\left(x+z_0\right)}{h\left(z_0\right)} .
    $$
    These functions belong to $C_0\left(\mathbb{R}^n\right)$.

    We note that this normalization argument is specific to the Asplund product.
Indeed, translations commute with the $\alpha$-sum for every
$\alpha\in[-\infty,0]$, but rescaling does not commute with
the $\alpha$-sum when $-\infty<\alpha<0$. More precisely,
\[
((af)\star_\alpha(bg))(x)
=
\sup_{y+z=x}
\left(a^\alpha f(y)^\alpha+b^\alpha g(z)^\alpha-1\right)^{1/\alpha},
\]
which is not, in general, a constant multiple of $f\star_\alpha g$.
Consequently, Theorem~\ref{thm: Plu-Ruz-0} does not extend to
functions with arbitrary maxima under the $\alpha$-sum.
For $\alpha=-\infty$, the normalization argument remains valid when
$f,g,h$ have the same maximum, but it fails in general when
their maxima are different.
\end{remark}

\subsection{The \texorpdfstring{$\alpha$}{α}-concave functions} \label{sec: alpha-concave-plu}
We extend our result to $C_\alpha(\R^n)$. In what follows, we restrict to $\alpha\in(-1/n,0)$, which guarantees integrability.

We now develop a comparison that recovers Theorem \ref{thm: Plu-Ruz-0} as $\alpha \to 0^-$. Define
$$
    B_\alpha f :=e^{-\operatorname{base}_\alpha(f)},
$$
which is a log-concave function for every $f\in C_\alpha(\mathbb R^n)$. It follows from the definition \eqref{def:alpha-sum-original} and \eqref{monotonicity} that
\begin{equation}
    \label{eq:03}
    B_\alpha (f\star_\alpha g) = B_\alpha f \star_0 B_\alpha g \leq B_\alpha f \star_\alpha B_\alpha g.
\end{equation}
We also have the pointwise inequality $B_\alpha f \leq f$. Indeed, using the fact that for any $u >0$, $u-1 \geq \log u$, we have
$f^\alpha-1\geq\ln f^\alpha=\alpha\ln f.$ Consequently,
\begin{equation}\label{eq:02}
    \|B_\alpha f\|_1\leq\|f\|_1.
\end{equation}
Importantly, the latter inequality can be reversed.

\begin{theorem} \label{thm:monotonicity-alpha-base}
    Let $\varphi: \mathbb{R}^n \rightarrow[0, \infty]$ be a lower semicontinuous convex function such that $\varphi(0)=0$ and $ \varphi (x) \to \infty$ as $|x| \to \infty$. If
    $
    -\frac{1}{n}<\alpha_1<\alpha_2 \leq 0,
    $
    then
    \begin{equation} \label{eq:monotonicity-alpha-base}
        \prod_{j=1}^n\left(1+j \alpha_1\right) \int_{\mathbb{R}^n}\left(1-\alpha_1 \varphi(x)\right)^{1 / \alpha_1} d x
        \leq
        \prod_{j=1}^n\left(1+j \alpha_2\right) \int_{\mathbb{R}^n}\left(1-\alpha_2 \varphi(x)\right)^{1 / \alpha_2} d x,
    \end{equation}
    where the expression at $\alpha_2=0$ is interpreted as $\int_{\mathbb{R}^n} e^{-\varphi(x)} \;\mathrm d x$.
\end{theorem}
\begin{proof}
    Define
    \[
        V(t) =\frac{|\{ x: \varphi (x)\leq t\}|}{t^n}, \qquad t > 0.
    \]
    It follows from Lemma \ref{lem:monotonicity-V} that $V$ is non-increasing.
    For $\alpha \in (-\frac{1}{n},0),$ applying Lemma \ref{lem:sublevel-set-representation} to the function $t \mapsto (1-\alpha t)^{1 / \alpha} $, we obtain
    \begin{align}
         \int_{\R^n} (1-\alpha  \varphi(x))^{1/\alpha} \;\mathrm dx
        = \int_{0}^{\infty}V(t) t^n (1 -\alpha t)^{1 / \alpha-1}  \;\mathrm dt.
    \end{align}
    Similarly, applying Lemma \ref{lem:sublevel-set-representation} to the function $t \to e^{-t}$, we get
    \[
    \int_{\mathbb{R}^n} e^{-\varphi(x)} \;\mathrm d x=\int_0^{\infty} V(t) t^ne^{-t} \;\mathrm d t.
    \]
    Define a measure $\mu_\alpha$ on $(0,\infty)$ with density
    \[
        p_\alpha (t) = \begin{cases}
            \frac{\prod_{j=1}^n(1+j \alpha)}{n!} t^n (1 -\alpha t)^{1 / \alpha-1} &-\frac{1}{n} <\alpha <0
            \\
            \frac{1}{n!} t^ne^{-t} &\alpha =0
        \end{cases}
    \]
    Then, $\mu_\alpha$ is a probability measure. For $\alpha=0$, this is the gamma integral. For $\alpha <0$,
    \begin{align}
        \int_0^{\infty} t^n(1-\alpha t)^{1 / \alpha-1} \;\mathrm d t
        & =(-\alpha)^{-n-1} B\left(n+1, -\frac{1}{\alpha}-n\right)
        \\
        &
        =(-\alpha)^{-n-1} \frac{\Gamma(n+1) \Gamma(-1 / \alpha-n)}{\Gamma(-1 / \alpha+1)}
        =\frac{n!}{\prod_{j=1}^n(1+j \alpha)} .
    \end{align}
    Substituting these formulas into \eqref{eq:monotonicity-alpha-base} and dividing both sides by $n!$, it suffices to prove
    \begin{equation} \label{eq:comparison}
        \int_{0}^{\infty} V(t) \left(p_{\alpha_1} (t) - p_{\alpha_2} (t) \right)\;\mathrm dt \leq 0.
    \end{equation}
    The desired inequality follows from Lemmas \ref{lem:measures-comparison} and \ref{lem:measures-comparison-non-increasing}.
    Therefore, it remains to verify the assumption of Lemma \ref{lem:measures-comparison}: there exists $c > 0$ such that
    \[
        p_{\alpha_1}(t) \leq p_{\alpha_2}(t) \quad \text { for a.e. } t \in(0, c),
    \quad \text{and} \quad
    p_{\alpha_1}(t) \geq p_{\alpha_2}(t) \quad \text { for a.e. } t \in(c, \infty) .
    \]
    Define the ratio
    \[
        R (t) = \frac{p_{\alpha_1} (t)}{p_{\alpha_2}(t)}, \qquad t>0.
    \]
    Its logarithmic derivative is
    $$
    \frac{d}{d t} \log R(t)=\frac{1-\alpha_2}{1-\alpha_2 t}-\frac{1-\alpha_1}{1-\alpha_1 t}=\frac{\left(\alpha_1-\alpha_2\right)(1-t)}{\left(1-\alpha_1 t\right)\left(1-\alpha_2 t\right)}.
    $$
    Thus, $R$ decreases on $( 0,1 )$ and increases on $( 1, \infty )$. Moreover,
    $$
        R(0) = \frac{\prod_{j=1}^n(1+j \alpha_1)}{\prod_{j=1}^n(1+j \alpha_2)} < 1, \qquad  \lim_{t \rightarrow \infty} R(t) = + \infty.
    $$
    Hence, $R$ decreases from a value below $1$ on $( 0,1 )$, and then increases strictly to $+\infty$ on $(1, \infty)$. It therefore crosses 1 exactly once, from below to above.
\end{proof}

\begin{corollary}\thlabel{cor:01}
For any $\alpha\in(-\frac{1}{n},0)$ and any $f\in C_\alpha(\mathbb R^n)$, we have
\begin{equation}\label{eq:01}
    \|f\|_1 \leq \frac{1}{\prod_{j=1}^n(1+j \alpha)}\left\|B_\alpha f\right\|_1.
\end{equation}
\end{corollary}
\begin{proof}
    Let $\varphi = \operatorname{base}_\alpha(f) = \frac{1 - f^\alpha}{\alpha}$. By the assumptions on $f$, the function $\varphi$ is convex, lower semicontinuous, $\min \varphi=\varphi(0)=0$, and $\varphi(x) \rightarrow \infty$ as $|x| \rightarrow \infty$.
    Moreover, $B_\alpha f = e^{-\varphi}$ and $f = (1-\alpha  \varphi)^{1/\alpha}$. Hence, \eqref{eq:01} follows from Theorem \ref{thm:monotonicity-alpha-base}.
\end{proof}

\begin{remark}
    The constant in Corollary \ref{cor:01} is sharp. To see this, let $K$ be a convex body containing the origin in its interior and set
    $$
    f(x)=\left(1-\alpha\|x\|_K\right)^{1 / \alpha} .
    $$
    Its convex base is the Minkowski functional $\varphi(x)=\|x\|_K$, and hence
    $$
    \{x: \varphi(x) \leq t\}=\left\{x:\|x\|_K \leq t\right\}=t K .
    $$
    Therefore,
    $$
    V(t)=\frac{|\{x: \varphi(x) \leq t\}|}{t^n}=\frac{|t K|}{t^n}=|K|.
    $$
    Thus, $V$ is constant, so equality holds in \eqref{eq:comparison} and consequently in \eqref{eq:monotonicity-alpha-base}. Hence, the constant cannot be improved.
\end{remark}

\plunruzalpha*

\begin{proof}
For $\alpha=0$, the assertion is exactly Theorem \ref{thm: Plu-Ruz-0}. Hence, assume $-1 / n<\alpha<0$.
Using Corollary \ref{cor:01}, inequality \eqref{eq:03}
and Theorem \ref{thm: Plu-Ruz-0}, we obtain
\begin{align*}
\|f\|_1\|f\star_\alpha g\star_\alpha h\|_1
&\leq \frac{1}{\prod_{j=1}^n(1+j \alpha)^2}\|B_\alpha f \|_1\|B_\alpha(f\star_\alpha g\star_\alpha h)\|_1\\
&\leq \frac{1}{\prod_{j=1}^n(1+j \alpha)^2}\|B_\alpha f\|_1\|B_\alpha f \star_\alpha B_\alpha g \star_\alpha B_\alpha h\|_1\\
&\leq \frac{c^n}{\prod_{j=1}^n(1+j \alpha)^2}\|B_\alpha f\star_\alpha B_\alpha g\|_1\|B_\alpha f \star_\alpha B_\alpha h\|_1
\end{align*}
We use the following elementary monotonicity: if $f_1 \leq f_2$ and $g_1 \leq g_2$, then
\begin{equation} \label{eq:claim-in-thm2}
    f_1 \star_\alpha g_1 \leq f_2 \star_\alpha g_2 .
\end{equation}
Indeed, for any $y,z \in \R^n$, since $\alpha <0$,
\[
    f_1^\alpha (y) + g_1^\alpha (z) -1 \geq f_2^\alpha (y) + g_2^\alpha (z) -1,
\]
Raising both sides to the power of $1/\alpha$, we prove \eqref{eq:claim-in-thm2}. Applying this monotonicity \eqref{eq:claim-in-thm2} together with the fact that $B_\alpha f \leq f, B_\alpha g \leq g$ and $B_\alpha h \leq h$, we obtain
\[
    \|f\|_1\|f\star_\alpha g\star_\alpha h\|_1
    \leq \frac{c^{n}}{\prod_{j=1}^n(1+j \alpha)^{2}}\| f\star_\alpha  g\|_1\|f \star_\alpha  h\|_1. \qedhere
\]
\end{proof}

As a consequence of the proof, using the same approach, one can prove the existence of constants in functional analogues of \cite[Theorem~4.1]{N-25}.
\begin{corollary}
    Let $\alpha \in (-1/n,0]$, let $m \in \mathbb{N}$, and let $f,g_1,\ldots, g_m \in C_\alpha (\R^n) $. Then,
    \begin{equation}
        \|f\|_1^{m-1} \| f \star_\alpha g_1 \star_\alpha \cdots \star_\alpha g_m \|_1 \leq \left( \frac{c^{n}}{\prod_{j=1}^n(1+j \alpha)^{2}} \right)^{m-1} \prod_{k=1}^m \|f \star_\alpha g_k\|_1,
    \end{equation}
    where $c$ is an absolute constant such that $c <26$.
\end{corollary}

\begin{proof}
    For $m=1$, the assertion is immediate. For $m \geq 2$, apply Theorem \ref{thm: Plu-Ruz-alpha} with $g=g_1$ and $h=g_2 \star_\alpha \cdots \star_\alpha g_m$,
    \[
        \|f\|_1^{m-1} \| f \star_\alpha g_1 \star_\alpha \cdots \star_\alpha g_m \|_1
        \leq
        \frac{c^n}{\prod_{j=1}^n(1+j \alpha)^2}
        \|f\|_1^{m-2} \| f \star_\alpha g_1 \|_1 \|f \star_\alpha g_2 \star_\alpha \cdots \star_\alpha g_m \|_1
    \]
    Iterating this estimate proves the assertion.
\end{proof}

\subsection{Reverse Pr\'ekopa--Leindler-type inequality}
In this subsection, we derive a reverse Pr\'ekopa--Leindler-type inequality for $\alpha$-sum with $\alpha \in (-1/n,0)$ from Corollary \ref{cor:01}.

We recall the reverse Brunn--Minkowski inequality proved by V. Milman \cite{VM-86,VM-86/87,VM-88} and also in \cite{BM-12}: For any convex bodies $A$ and $B$ in $\mathbb{R}^n$, there exist linear volume-preserving maps $u_1,u_2 \in SL_n$ such that with some absolute constant $C$,
\begin{equation}
    \label{eq:Reverse-BM}
    |u_1(A)+u_2(B)|^{1 / n} \leq C\left(|A|^{1 / n}+|B|^{1 / n}\right).
\end{equation}
The functional analogue of \eqref{eq:Reverse-BM} for log-concave functions was proved by Klartag and V.~Milman \cite[Theorem 1.3]{KM-05}: Let $f_1,f_2$ be log-concave functions that attain their maximum value of 1 at the origin. There exist  $u_1,u_2 \in SL_n$ such that $\Tilde{{f}_i} = f_i \circ u_i $ for $i =1,2$,
\begin{equation}
    \label{eq: reverse-PL_0}
    \|\Tilde{{f}_1} \star_0 \Tilde{{f}_2}\|_1^{\frac{1}{n}}\leq C \left(\| {f}_1\|_1^{\frac{1}{n}}+\| {f}_2\|_1^{\frac{1}{n}}\right),
    \end{equation}
    where $C>0$ is an absolute constant.

\begin{theorem}
    \label{thm: reverse-PL_alpha}
    Fix $\alpha \in (-1/n,0] $ and let $f_1,f_2 \in C_\alpha (\R^n)$. Then, there exist $u_1,u_2 \in SL_n$ such that $\Tilde{f}_i = f_i \circ u_i $ for $i =1,2$,
    \begin{equation}
    \label{eq: reverse-PL_alpha}
        \|\Tilde{{f}_1} \star_\alpha \Tilde{{f}_2}\|_1^{\frac{1}{n}}
        \leq
        \frac{C}{\left(\prod_{j=1}^n(1+j \alpha)\right)^{1/n}}\left(\| {f}_1\|_1^{\frac{1}{n}}+\| {f}_2\|_1^{\frac{1}{n}}\right),
    \end{equation}
    where $C>0$ is an absolute constant.
\end{theorem}

\begin{proof}
    We may assume $-1 / n<\alpha<0$. Apply \eqref{eq: reverse-PL_0} to the log-concave functions $B_\alpha f_1$ and $B_\alpha f_2$, and choose $u_1, u_2 \in S L_n$ accordingly. Note that $B_\alpha\left(f_i \circ u_i\right)=\left(B_\alpha f_i\right) \circ u_i$.
    Using Corollary \ref{cor:01}, and definition \eqref{def:alpha-sum-original}, we obtain
    \begin{align}
        \| \Tilde{{f}_1} \star_\alpha \Tilde{{f}_2}\|_1^{\frac{1}{n}}
        &\leq
        \left(\frac{1}{\prod_{j=1}^n(1+j \alpha)}\right)^{1/n}
        \| B_\alpha (\Tilde{{f}_1} \star_\alpha \Tilde{{f}_2} )\|_1^{\frac{1}{n}}
        \\
        &=
        \left(\frac{1}{\prod_{j=1}^n(1+j \alpha)}\right)^{1/n}
        \| B_\alpha \Tilde{{f}_1}   \star_0 B_\alpha \Tilde{{f}_2} \|_1^{\frac{1}{n}} .
    \end{align}
    Applying \eqref{eq: reverse-PL_0}, we get
    \begin{align}
        \| \Tilde{{f}_1} \star_\alpha \Tilde{{f}_2}\|_1^{\frac{1}{n}}
        &\leq
        \frac{C}{\left(\prod_{j=1}^n(1+j \alpha)\right)^{1/n}}
        \left( \|B_\alpha \Tilde{{f}_1}  \|_1^{1/n} + \|B_\alpha \Tilde{{f}_2}  \|_1^{1/n}\right)
        \nonumber
        \\
        &\leq
        \frac{C}{\left(\prod_{j=1}^n(1+j \alpha)\right)^{1/n}}
        \left( \|\Tilde{f}_1 \|_1^{1/n} +\|\Tilde{f}_2 \|_1^{1/n} \right)
        \\
        &=
        \frac{C}{\left(\prod_{j=1}^n(1+j \alpha)\right)^{1/n}}
        \left( \|{f}_1 \|_1^{1/n} +\|{f}_2 \|_1^{1/n} \right).
    \end{align}
    Here, the second inequality uses \eqref{eq:02}, and the last equality follows from $\operatorname{det} u_i=1$.
\end{proof}

\begin{remark}
    We note that for log-concave functions, taking the limit $\alpha \to 0^-$ in \eqref{eq: reverse-PL_alpha} recovers the reverse Pr\'ekopa–Leindler inequality established in \cite{KM-05}.
\end{remark}

\begin{remark}
    Bobkov and the first-named author \cite{BM11-cras} obtained an entropic lifting for log-concave densities, showing that a reverse entropy power inequality holds in this setting.
    Subsequently, in \cite{BM-12}, they obtained a more general result for $\alpha$-concave densities with $\alpha<0$, analogous to Theorem~\ref{thm: reverse-PL_alpha}. Finally, in \cite{BM-13}, they showed that a reverse entropy power inequality cannot hold if $\alpha=-1/n$.
\end{remark}

\section{Rogers--Shephard-type inequality for functions} \label{sec:RS}

Let us begin with a functional extension of \eqref{eq:RS-bodies} for $\alpha$-concave functions, due to Colesanti \cite{C-06}. Recall that $\Bar{f} (x) = f(-x)$. The difference function of order $\alpha \in [-\infty,0]$ of a function $f$ is
\begin{equation}
    \label{df:alpha-sum-Colestanti}
    \Delta_\alpha f(x) = \sup_{2x = y+z} M_\alpha (f(y),\Bar f(z)),
\end{equation}
where $M_\alpha (a,b)$ is the mean of order $\alpha$ defined by
\[
     M_\alpha (a,b) =
     \begin{cases}
        \sqrt{ab} \quad & \alpha = 0,
        \\
        \left(\frac{a^\alpha+b^\alpha}{2} \right)^{1/\alpha} \quad & -\infty <\alpha < 0,
        \\
        \min\{a,b\} \quad& \alpha = -\infty.
     \end{cases}
\]

Let $\mathcal{C} (n,\alpha)$ be the smallest possible constant such that for any integrable  $f \in C_\alpha (\R^n)$ with full-dimensional support,
\begin{equation}
    \label{const:alpha-sum-Colestanti}
    \| \Delta_\alpha f \|_1 \leq \mathcal{C} (n,\alpha)\|f\|_1.
\end{equation}
We note that in \cite{C-06}, the functions under consideration may be unbounded. However, in determining $\mathcal{C}(n, \alpha)$, it is enough to consider functions in $C_\alpha\left(\mathbb{R}^n\right)$. Indeed, if $f$ is bounded, then translating a point at which $f$ attains its maximum to the origin and dividing $f$ by this maximum yields a function in $C_\alpha\left(\mathbb{R}^n\right)$. The inequality is invariant under translations and rescaling.
If $f$ is unbounded, for each $m>0$, define $f_m=\min \{f, m\}.$ Then, $f_m$ is also $\alpha$-concave and converges pointwise to $f$ as $m \to \infty$. After translation and scaling, we again obtain a function in $C_\alpha (\R^n)$. Moreover, $\Delta_\alpha f_m $ converges to $ \Delta_\alpha f$ pointwise as $m \to \infty$. Hence, applying the monotone convergence theorem to both sides gives the desired inequality for $f$.

It was proved in  \cite{C-06} that
    $\mathcal{C} (n,\alpha) \leq \frac{1}{2^{n+1 / \alpha}}\binom{2 n}{n}.$
Moreover, the constant is known precisely in the limiting cases,
\begin{equation}
\label{eq:C-06-sharp}
\mathcal{C} (n,-\infty) = \frac{1}{2^n} \binom{2n}{n}\quad \text{and} \quad
\mathcal{C} (n,0) = 2^n.
\end{equation}
In addition, in dimension one, the constants have been determined explicitly
$$
   \mathcal{C} (1, \alpha)=
    \begin{cases}
        2,  & \quad \alpha \in(-1,0),
        \\
        \frac{1}{2^{1 / \alpha}} & \quad \alpha \in(-\infty,-1] .
    \end{cases}
$$
The operation used in this work differs from Colesanti's difference function. The mean $M_\alpha(a, b)$, and hence $\Delta_\alpha f$, is non-decreasing in $\alpha$, whereas Proposition \ref{prop:monotonicity} shows that $f \star_\alpha \bar f$ is non-increasing in $\alpha$. At $\alpha=-\infty$, the two operations agree after dilation,
$$
\Delta_{-\infty} f(x / 2)=(f \oplus \bar{f})(x) .
$$
Thus, the first equality in \eqref{eq:C-06-sharp} can be written as
\begin{equation}
    \label{eq:RS-infty}
    \| f \oplus \Bar{f} \|_1 \leq \binom{2n}{n} \| f\|_1,
\end{equation}
where equality occurs precisely when every level set $U_t(f)$ is a simplex for almost every $t > 0$.

The case of log-concave functions under the Asplund product was proved in \cite[Theorem~2.2]{AGJV-16}, where the equality is attained when the function is a characteristic function of a simplex. Specially, for any $f \in C_0(\R^n)$ with full-dimensional support, we have
\begin{equation}
    \label{eq:RS-ineq-Logconcave}
    \|f \star_0 \Bar{f} \|_1 \leq \binom{2n}{n} \|f\|_1.
\end{equation}
For $n\geq 2$, equality holds if and only if $f$ is a characteristic function of a simplex.

The next result interpolates between \eqref{eq:RS-infty} and \eqref{eq:RS-ineq-Logconcave} for the $\alpha$-sum. The inequality follows immediately from \eqref{eq:RS-infty} and Proposition \ref{prop:monotonicity}. We include the proof to determine the equality cases.

\RSAlphaSum*

\begin{proof}
    By Proposition \ref{prop:monotonicity}, the layer-cake representation \eqref{mix-integral-for} and the level-set identity \eqref{sum-Uf}, we obtain
    \begin{align}
        \|f \star_{\alpha} \Bar{f}\|_1  &\leq
		\|f \oplus \Bar{f}\|_1 \label{eq:mono-rs}
        = \int_0^1 |U_t(f\oplus \Bar{f} ) | \;\mathrm dt
		= \int_0^1 |U_t(f) + U_t(\Bar{f})| \;\mathrm dt
        \\
        &= \int_0^1 |U_t(f) - U_t({f})| \;\mathrm dt.
	\end{align}
    Since $ U_t(f) $ are compact, convex sets for all $ t >0 $, we apply the classical Rogers--Shephard inequality \eqref{eq:RS-bodies} and use the layer-cake representation \eqref{mix-integral-for},
	\begin{align}
        \|f \star_{\alpha} \Bar{f}\|_1
		&\leq \int_0^1 |U_t(f) - U_t(f)| \;\mathrm dt
		\leq \binom{2n}{n} \int_0^1 |U_t(f)|\;\mathrm dt \label{eq:RS-0-eqcase}
		= \binom{2n}{n} \|f\|_1.
	\end{align}
    For the equality case, we first treat the case $n\geq 2$. For $\alpha=0$, this is simply the equality already obtained.  Hence, we need to check only for $\alpha <0$.
    If $f$ is the characteristic function of a simplex, \eqref{RS-QC-eq-intro} reduces to the classical Rogers--Shephard inequality \eqref{eq:RS-bodies}, and equality holds.

    Conversely, assume that equality holds in \eqref{RS-QC-eq-intro}.
    Thus, the equality in the first inequality \eqref{eq:mono-rs} gives $f \star_{\alpha} \Bar{f}  = f \oplus \Bar{f}$  almost everywhere, while equality in the second inequality \eqref{eq:RS-0-eqcase} implies that $U_t(f)$ is a simplex for almost every $t \in [0,1]$.
    We claim that almost everywhere on $\R^n$,
    \begin{equation}
        B_\alpha f \star_{0} \overline{B_\alpha f}  = B_\alpha f  \oplus \overline{B_\alpha f}. \label{eq:pt-Karamata-zero}
    \end{equation}
    For almost every $x$, the equality $f \star_\alpha \bar{f}=f \oplus \bar{f}$, together with the fact that $1 / \alpha<0$, is equivalent to

    \begin{align}
        (f \star_{\alpha} \Bar{f} )(x)  &= (f \oplus \Bar{f} )(x) \label{eq:pt-Karamata-alpha}
        \\
        \sup_y \left(f^\alpha (y)+ f^\alpha (y-x) -1 \right)^{1/\alpha} &= \sup_y \min \{f (y), f (y-x) \}
        \\
        \inf_y (f^\alpha (y)+ f^\alpha (y-x) -1 )&= \inf_y \max \{f^\alpha (y), f^\alpha (y-x) \}
        \\
        \sup_y (1- f^\alpha (y)+ 1- f^\alpha (y-x))  &= \sup_y \min \{1- f^\alpha (y), 1- f^\alpha (y-x) \}.
    \end{align}
    Here, the second step uses that $\frac{1}{\alpha} < 0$. Since $f^\alpha = 1-\alpha \operatorname{base}_\alpha f$ and $B_\alpha f = e^{- \operatorname{base}_\alpha f}$, the last equality is equivalent to
    \begin{align}
        \inf_y \left((\operatorname{base}_\alpha f) (y) + (\operatorname{base}_\alpha f) (y-x) \right) &= \inf_y \max \{(\operatorname{base}_\alpha f) (y) ,(\operatorname{base}_\alpha f) (y-x)  \}
        \\
        \sup_y \left(B_\alpha f(y) \cdot  B_\alpha f(y-x) \right) &= \sup_y \min \{ B_\alpha f(y), B_\alpha f(y-x) \}
        \\
        B_\alpha f \star_{0} \overline{B_\alpha f}  &= B_\alpha f  \oplus \overline{B_\alpha f} .
    \end{align}
    The function $B_\alpha f$ is log-concave and continuous on its full-dimensional support. Moreover, its level set
    $$
        U_{t} \left( B_\alpha f\right) = \{ B_\alpha f \geq t\} = \{  \operatorname{base}_\alpha f\leq -\log t\} =\{ f\geq (1+\alpha \log t)^{1 / \alpha}
        \}.
    $$
    are simplices for almost every $t \in [0,1]$. Hence, $B_\alpha f$ satisfies the equality case of \eqref{eq:RS-ineq-Logconcave}. Indeed, repeating the inequality argument with $f$ replaced by $B_\alpha f$ and $\alpha = 0$,
    \[
        \left\|B_\alpha f \star_0 \overline{B_\alpha f}\right\|_1=\left\|B_\alpha f \oplus \overline{B_\alpha f}\right\|_1=\binom{2 n}{n}\left\|B_\alpha f\right\|_1.
    \]
    It follows from the equality case of \eqref{eq:RS-ineq-Logconcave} that $B_\alpha f$ has to be a characteristic function of a simplex, and therefore $f$ is as well.

    We now consider $n=1$.  First, suppose that  $f$ is monotone on its support. Using translation and reflection, if necessary, we may assume that the support of $f$ belongs to $[0, \infty)$ and $f$ is non-increasing on $[0, \infty)$. For $\alpha <0$,
    we have
    \begin{equation}
    \label{eq:explicit-form-some-alpha-sum-in-pf}
        (f \star_\alpha \Bar f) (x)
        =\sup_{ y +z =x } (f(y)^\alpha + f(-z)^\alpha -1)^{1/\alpha}
        .
    \end{equation}
    At $\alpha=0$, the expression in parentheses is replaced by $f(y) f(-z)$. Only pairs satisfying $y \geq 0$ and $z \leq 0$ need to be considered. Otherwise $f(y) = 0$ or $f(-z) = 0$, and by convention, the corresponding value in the supremum is zero.
\begin{enumerate}[wide =0pt, labelwidth= .9cm, leftmargin=1.55cm, label=\textbf{Case \arabic*}:]
        \item Suppose that $x \geq 0$. Hence, $ y = x-z \geq x$. Since $f$ is non-increasing on $[0,\infty)$, $f(y) \leq f(x).$ Since $f(-z) \leq 1$, we obtain
    \[
        (f(y)^\alpha + f(-z)^\alpha -1)^{1/\alpha} \leq f(x).
    \]
    Taking $y= x$ and $z= 0$, we obtain equality and so $(f \star_\alpha \Bar f) (x) = f(x)$.
    \item Suppose that $x < 0$. Hence, $ -z = -x+y \geq -x$. Since $f$ is non-increasing on $[0,\infty)$, $f(-z) \leq f(-x).$ Since $f(y) \leq 1$, we obtain
    \[
        (f(y)^\alpha + f(-z)^\alpha -1)^{1/\alpha} \leq f(-x).
    \]
    Taking $y= 0$ and $z= x$, we obtain equality and so $(f \star_\alpha \Bar f) (x) = f(-x)$.
    \end{enumerate}
    Hence,
    \[
        (f  \star_\alpha \bar f) (x) = \begin{cases}
            f(x) & x\geq 0,
            \\
            f(-x) & x< 0.
        \end{cases}
    \]
    Therefore, $\|f \star_\alpha \bar{f}\|_1=2\|f\|_1$. The same conclusion follows at $\alpha=0$ from the product formula.

    Conversely, assume that the equality in \eqref{RS-QC-eq-intro} holds in dimension one. Using translation if needed, we may assume that
    $
        U_1 (f) = [0,x_1], $ for some $ x_1 \geq 0.
    $
    As in the case $n\geq 2$, the equality forces \eqref{eq:pt-Karamata-zero} to hold for almost every $x \in \R$, that is,
    \begin{equation}
        B_\alpha f \star_{0} \overline{B_\alpha f}  = B_\alpha f  \oplus \overline{B_\alpha f}. \label{eq:pt-Karamata-zero-second-times}
    \end{equation}
    Here, $B_0 f = f$, and
    \[
        U_1 (B_\alpha (f)) = [0,x_1], \quad x_1 \geq 0.
    \]
    It is enough to prove that $B_\alpha f$ is monotone on its support.
    Suppose, toward a contradiction, that $B_\alpha f$ is not monotone on its support. Hence, by log-concavity, there exist
    \begin{center}
        $x_0< 0 \leq x_1 <x_2$\quad  such that  \quad$0<B_\alpha f(x_0),B_\alpha f(x_2) < 1.$
    \end{center}
    Choose $x \in (x_1,x_1 + \min\{-x_0,x_2-x_1\})$ at which \eqref{eq:pt-Karamata-zero-second-times} holds. Thus,
    $B_\alpha f(y)$ and $B_\alpha f(y-x)$ are non-decreasing on $(-\infty,x_1]$, and non-increasing on $[x,\infty)$. Therefore,
    \[
        \bigl(B_\alpha f \star_0 \overline{B_\alpha f}\bigr)(x)
        =
        \sup_{y\in\R} B_\alpha f(y)\,B_\alpha f(y-x)
        =
        \sup_{y\in[x_1,x]} B_\alpha f(y)\,B_\alpha f(y-x).
    \]
    By continuity, the supremum is attained at some $y_0\in[x_1,x]$, that is,
    \[
        \bigl(B_\alpha f \star_0 \overline{B_\alpha f}\bigr)(x)
        =
        B_\alpha f(y_0)\,B_\alpha f(y_0-x).
    \]
    Using \eqref{eq:pt-Karamata-zero-second-times}, we obtain
    \begin{align}
        B_\alpha f(y_0)\,B_\alpha f(y_0-x) &\leq \min\{B_\alpha f(y_0),\,B_\alpha f(y_0-x)\} \leq (B_\alpha f  \oplus \overline{B_\alpha f}) (x)
        \\
        &=  \bigl(B_\alpha f \star_0 \overline{B_\alpha f}\bigr)(x) =B_\alpha f(y_0)\,B_\alpha f(y_0-x)
        .
    \end{align}
    Hence,
    \begin{align} \label{eq:pf-n1}
        B_\alpha f(y_0)\,B_\alpha f(y_0-x)
        =
        \min\{B_\alpha f(y_0),\,B_\alpha f(y_0-x)\}.
    \end{align}
    If $y_0 \in\left(x_1, x\right)$, then the choice of $x$ ensures that $0<B_\alpha f(y_0),B_\alpha f(y_0-x)<1$. This contradicts \eqref{eq:pf-n1}.
    Therefore, $y_0$ is either $x_1$ or $x$, and it follows that
    \[
        \bigl(B_\alpha f \oplus \overline{B_\alpha f}\bigr)(x)
        =
        \max\{B_\alpha f(x),\,B_\alpha f(x_1-x)\}.
    \]
\begin{enumerate}[wide =0pt, labelwidth= .9cm, leftmargin=1.55cm, label=\textbf{Case \arabic*}:]
        \item Suppose that $\bigl(B_\alpha f \oplus \overline{B_\alpha f}\bigr)(x)=B_\alpha f(x)$. For any $ \varepsilon \in (0,x-x_1) $, we have
        \[
            x-\varepsilon=\frac{\varepsilon}{x-x_1} \cdot x_1+\left(1-\frac{\varepsilon}{x-x_1}\right) x.
        \]
        Therefore, by log-concavity of $B_\alpha f$ and the fact that $B_\alpha f (x_1) = 1$,
        \[
            B_\alpha f\left(x-\varepsilon\right) \geq  (B_\alpha f (x))^{1-\frac{\varepsilon}{x-x_1}} > B_\alpha f (x).
        \]
        Also, since $0$ is an interior point of the support of $B_\alpha f$,
        $$
            \lim_{\varepsilon \to 0^+} B_\alpha f (-\varepsilon) = B_\alpha f (0) = 1 > B_\alpha f (x).
        $$
        Hence, for small enough $\varepsilon >0$,
        \[
            B_\alpha f(x) < \min\{B_\alpha f(-\varepsilon),\,B_\alpha f(x-\varepsilon)\} \leq \bigl(B_\alpha f \oplus \overline{B_\alpha f}\bigr)(x) = B_\alpha f (x),
        \]
         which is a contradiction.
         \item Suppose that $\bigl(B_\alpha f \oplus \overline{B_\alpha f}\bigr)(x)=B_\alpha f(x_1-x)$. For any $ \varepsilon \in (0,x-x_1) $, we have
        \[
            x_1-x+\varepsilon=\frac{\varepsilon}{x-x_1} \cdot 0+\left(1-\frac{\varepsilon}{x-x_1}\right) (x_1-x).
        \]
        Therefore, by log-concavity of $B_\alpha f$ and the fact that $B_\alpha f (0) = 1$,
        \[
            B_\alpha f\left(x_1-x+\varepsilon\right) \geq  (B_\alpha f (x_1-x))^{1-\frac{\varepsilon}{x-x_1}} > B_\alpha f (x_1-x).
        \]
        Also, since $x_1$ is an interior point of the support of $B_\alpha f$,
        $$
            \lim_{\varepsilon \to 0^+} B_\alpha f (x_1+\varepsilon) = B_\alpha f (x_1) = 1 > B_\alpha f (x_1-x).
        $$
        Hence, for small enough $\varepsilon>0$,
        \begin{align}
            B_\alpha f(x_1-x) &< \min\{B_\alpha f(x_1+\varepsilon),\,B_\alpha f(x_1-x+\varepsilon)\} \leq \bigl(B_\alpha f \oplus \overline{B_\alpha f}\bigr)(x)
            \\
            &= B_\alpha f (x_1-x),
        \end{align}
         which is a contradiction.
    \end{enumerate}
    Combining these two cases, this shows that $f$ must be monotone on its support. \qedhere
\end{proof}

\section{Functional sum-difference inequalities}
\label{sec:sumdiff}

\LitvakFuncQuasi*

\begin{proof}
    Using the layer-cake representation \eqref{mix-integral-for} and the level set identity \eqref{sum-Uf}, we obtain
    \begin{align*}
        \| f \oplus g\|_1 &=\int_0^1|U_t(f \oplus {g})| \;\mathrm dt
        \\
        &=
        \int_0^1|U_t(f)+ U_t({g})| \;\mathrm dt .
    \end{align*}
    Now, we apply Litvak's observation \eqref{eq:Litvak Ob} and use the layer-cake representation \eqref{mix-integral-for},
    \begin{align*}
        \| f \oplus g\|_1
        &\leq
        \frac{1}{2^n}\binom{2 n}{n} \int_0^1|U_t(f) -  U_t({g})| \;\mathrm dt
        \\
        &=\frac{1}{2^n}\binom{2 n}{n} \| f \oplus \Bar{g}\|_1.
    \end{align*}
    If $f = \bar g$, and its superlevel sets are simplices for almost every $t>0$, then $ -U_t ( g) = U_t (f).$ Therefore, the equality condition in Litvak's observation is satisfied for almost every level $t$, and hence the inequality above is an equality.
\end{proof}

The following estimate is the two-family analogue of the argument in the work of Kotrbat\'y and Mouamine \cite{KM-26}.

\begin{lemma} \label{lem:KM-lp}
    Let $\varphi_1,\varphi_2 : \R^n \to [0,\infty]$  be lower semicontinuous, convex functions satisfying
    \begin{center}
        $\varphi_1(0) = \varphi_2(0) = 0$\quad and \quad$\varphi_1(x), \varphi_2 (x) \to + \infty$ as $|x| \to \infty$.
    \end{center}
    For $r >0$ and $0 \leq t\leq 1$, define
    \[
        C_t(r) = \{x: \varphi_1(x) \leq rt\} + \{x: \varphi_2(x) \leq r(1-t)\}.
    \]
    Then,
    \[
        \frac{1}{n+1} \left|\bigcup_{0 \leq t \leq 1} C_t(r) \right|
        \leq
        \int_0^1\left|C_t(r)\right| \;\mathrm d t
        \leq
        \left|\bigcup_{0 \leq t \leq 1} C_t(r)\right|.
    \]
\end{lemma}

\begin{proof}
    Denote
    \[
        S_r:=\bigcup_{0 \leq t \leq 1} C_t(r).
    \]
    We first show that $S_r$ is compact and convex.
    Define
$$
\widetilde{\Sigma}_r=\left\{(a, b, t) \in \mathbb{R}^n \times \mathbb{R}^n \times[0,1]: \varphi_1(a) \leq r t, \varphi_2(b) \leq r(1-t)\right\} .
$$
The functions
$$
F_1(a, b, t) = \varphi_1(a)-r t, \quad F_2(a, b, t) = \varphi_2(b)-r(1-t)
$$
are lower semicontinuous. Since $ \widetilde{\Sigma}_r = \{F_1\leq 0\} \cap \{ F_2 \leq 0\}$,  the set $\widetilde{\Sigma}_r$ is closed. Moreover, since
$$
\widetilde{\Sigma}_r \subset \left\{x \in \mathbb{R}^n: \varphi_1(x) \leq r\right\} \times \left\{x \in \mathbb{R}^n: \varphi_2(x) \leq r\right\} \times[0,1],
$$
the set $\widetilde{\Sigma}_r$ is compact. It is also convex.  To see this,
    let $(a_1,b_1,t_1),(a_2,b_2,t_2) \in \widetilde{\Sigma}_r$. Then,
    \[
    \varphi_1(a_i) \leq rt_i \quad \varphi_2(b_i) \leq r(1-t_i) \qquad \text{for } i =1,2.
    \]
    Let $\lambda \in [0,1]$. Then,
    \[
        \varphi_1 (\lambda a_1 +(1-\lambda)a_2) \leq \lambda \varphi_1\left(a_1\right)+(1-\lambda) \varphi_1\left(a_2\right) \leq r (\lambda t_1+(1-\lambda) t_2),
    \]
    and
    \[
        \varphi_2 (\lambda b_1+(1-\lambda) b_2) \leq \lambda \varphi_2\left(b_1\right)+(1-\lambda) \varphi_2\left(b_2\right) \leq r(1-(\lambda t_1+(1-\lambda) t_2)).
    \]
    Hence, $\widetilde{\Sigma}_r$ is convex.
    Consider the linear maps
    $$
    P(a, b, t)=a+b, \quad Q(a, b, t)=(a+b, t) .
    $$
    Then
    $$
    P\left(\widetilde{\Sigma}_r\right)=S_r
    , \qquad
    Q\left(\widetilde{\Sigma}_r\right)=\Sigma_r:=\left\{(x, t) \in \mathbb{R}^n \times[0,1]: x \in C_t(r)\right\} .
    $$
    Consequently, both $S_r$ and $\Sigma_r$ are compact and convex.
    For $x \in S_r$, define
    $$
    I_x=\left\{t \in[0,1]: x \in C_t(r)\right\} .
    $$
    Since $\Sigma_r$ is compact and convex, $I_x$, being a one-dimensional section of $\Sigma_r$, is a compact interval.

    Now, we move on to the inequalities. The right inequality is immediate. Indeed, since $C_t(r) \subset S_r$ for any $t \in [0,1]$, we have
    \[
        \int_0^1\left|C_t(r)\right| \;\mathrm d t
        \leq
        \int_0^1\left|S_r\right| \;\mathrm d t
        =
        \left|S_r\right|.
    \]
    We next prove the left inequality. Recall the Minkowski functional of $S_r$,
    $$
    \|x\|_{S_r}:=\inf \left\{\lambda \geq 0: x \in \lambda S_r\right\}.
    $$
    Because $0 \in S_r$, we have $0 \leq \|x\|_{S_r} \leq 1$ for $x \in S_r$. We claim that for $x \in S_r$,
    \[
        |I_x| \geq 1-\|x\|_{S_r}
    \]
   \begin{enumerate}[wide =0pt, labelwidth= .9cm, leftmargin=1.55cm, label=\textbf{Case \arabic*}:]
        \item Let $x \in S_r \setminus\{0\}$. Since  $\frac{x}{\|x\|_{S_r}} \in S_r = P\left(\widetilde{\Sigma}_r\right)$, there exist $a,b \in \R^n$ and $t_0 \in [0,1]$ such that
    \[
        \frac{x}{\|x\|_{S_r}}=a+b, \quad \varphi_1(a) \leq r t_0, \quad \varphi_2(b) \leq r\left(1-t_0\right).
    \]
    Since $\varphi_1(0)=\varphi_2(0)=0$, the convexity of $\varphi_1$ and $\varphi_2$ gives
    $$
    \varphi_1(\|x\|_{S_r} a) \leq \|x\|_{S_r} \varphi_1(a) \leq r \|x\|_{S_r} t_0
    , $$
    and
    $$
    \varphi_2(\|x\|_{S_r} b) \leq \|x\|_{S_r} \varphi_2(b) \leq r \|x\|_{S_r}\left(1-t_0\right) .
    $$
    Let $t \in\left[\|x\|_{S_r} t_0, 1-\|x\|_{S_r}\left(1-t_0\right)\right]$. Then,
    \begin{center}
        $\varphi_1(\|x\|_{S_r} a) \leq rt$\quad and \quad$ \varphi_2(\|x\|_{S_r} b) \leq r \left(1-t\right)$.
    \end{center}
    Since $x = \|x\|_{S_r}a+\|x\|_{S_r}b$, by definition of $C_t(r)$, we get that $x \in C_t(r)$ and so $t \in I_x$. Hence,
    \[
        \left[\|x\|_{S_r} t_0, 1-\|x\|_{S_r}\left(1-t_0\right)\right] \subset I_x.
    \]
    It follows that
    \[
        |I_x| \geq 1-\|x\|_{S_r}.
    \]
    \item For $x=0$, since $\varphi_1(0)=\varphi_2(0)=0$, we have $0 \in C_t(r)$ for every $t \in[0,1]$. Thus,
$
I_0=[0,1]
$
and
$
\left|I_0\right|=1=1-\|0\|_{S_r} .
$
    \end{enumerate}
This proves the claim, that is,
$$
\left|I_x\right| \geq 1-\|x\|_{S_r}, \quad x \in S_r.
$$
Since $\Sigma_r$ is compact and therefore measurable, Tonelli's theorem yields
$$
\begin{aligned}
\int_0^1\left|C_t(r)\right|  \;\mathrm{d} t
& =\int_{\mathbb{R}^n} \int_0^1 {1}_{\Sigma_r}(x, t) \;\mathrm{d} t \;\mathrm{d} x
=\int_{S_r}\left|I_x\right|  \;\mathrm{d} x
& \geq \int_{S_r}1-\|x\|_{S_r}  \;\mathrm{d} x.
\end{aligned}
$$
Using layer-cake representation, we obtain
\begin{align}
\int_{S_r}1-\|x\|_{S_r}  \;\mathrm{d} x &= \int_{S_r}\int_0^1 {1}_{\{\|x\|_{S_r}  \leq s\}} \;\mathrm{d} s \;\mathrm{d} x
 =\int_0^1\left|\left\{x \in S_r: \|x\|_{S_r}  \leq s\right\}\right| \mathrm{d} s \\
& =\int_0^1\left|s S_r\right| \mathrm{d} s
 =\left|S_r\right| \int_0^1 s^n \mathrm{~d} s
 =\frac{\left|S_r\right|}{n+1}.
\end{align}
This proves the left inequality.
\end{proof}

\begin{theorem}
    \label{prop:litvak-fun}
    Let $\alpha \in (-\frac{1}{n},0]$, and let $f,g \in C_{\alpha}(\R^n)$. Then
    \begin{equation} \label{eq:Litvak-alpha}
        \| f \star_\alpha g\|_1 \leq \min \left\{\frac{n+1}{2^n}\binom{2 n}{n},\frac{2^n}{\prod_{j=1}^n(1+j \alpha)} \right\} \| f \star_\alpha \Bar{g}\|_1.
    \end{equation}
\end{theorem}

\begin{proof}
    For $\alpha=0$, set $B_0 f=f$ and interpret the first comparison below as equality.
    Using Corollary \ref{cor:01} and definition \eqref{def:alpha-sum-original}, we have
    \[
        \| f \star_\alpha g\|_1 \leq \frac{1}{\prod_{j=1}^n(1+j \alpha)}\left\|B_\alpha (f \star_\alpha g)\right\|_1
          = \frac{1}{\prod_{j=1}^n(1+j \alpha)}\left\|B_\alpha f  \star_0 B_\alpha g \right\|_1.
    \]
    We apply Pr\'ekopa--Leindler inequality  \eqref{eq:pli}, and the result by Colesanti \eqref{eq:C-06-sharp},
    \begin{align}
        \left\|B_\alpha f  \star_0 B_\alpha g \right\|_1
        &\leq
        \left\|\left(\frac{1}{2} \cdot (B_\alpha f  \star_0 B_\alpha g)\right) \star_0\left(\frac{1}{2} \cdot \overline{B_\alpha f  \star_0 B_\alpha g}\right)\right\|_1
        \\
        &=\left\| \left(\frac{1}{2} \cdot B_\alpha f \right) \star_0  \left(\frac{1}{2} \cdot B_\alpha g \right) \star_0 \left(\frac{1}{2} \cdot \overline{B_\alpha f} \right)\star_0 \left(\frac{1}{2} \cdot \overline{B_\alpha g}\right)\right\|_1
        \\
        &=\left\|\left(\frac{1}{2} \cdot (B_\alpha f  \star_0 \overline{B_\alpha g})\right) \star_0\left(\frac{1}{2} \cdot \overline{B_\alpha f  \star_0 \overline{B_\alpha g}}\right)\right\|_1
        \\
        &= \|\Delta_{0} \left( B_\alpha f  \star_0 \overline{B_\alpha g}\right)\|_1
        \\
        &\leq 2^n \|  B_\alpha f  \star_0 \overline{B_\alpha g}\|_1.
    \end{align}
    Combining the previous inequalities with \eqref{eq:02}, we have
    \begin{align}
        \| f \star_\alpha g\|_1 &\leq \frac{2^n}{\prod_{j=1}^n(1+j \alpha)}  \|  B_\alpha f  \star_0 \overline{B_\alpha g}\|_1
        = \frac{2^n}{\prod_{j=1}^n(1+j \alpha)}  \|  B_\alpha (f  \star_\alpha \Bar g)\|_1
        \\
        &\leq \frac{2^n}{\prod_{j=1}^n(1+j \alpha)}  \|  f  \star_\alpha \Bar g\|_1.
    \end{align}
    Now, we prove the first estimate. Let $\varphi_1=\operatorname{base}_\alpha f$ and $\varphi_2=\operatorname{base}_\alpha g$. For $r >0$ and $t \in [0,1]$,  define
    \[
        C_t^+(r) = \{x: \varphi_1(x) \leq rt\} + \{x: \varphi_2(x) \leq r(1-t)\}
    \]
    and
    \[
        C_t^-(r) = \{x: \varphi_1(x) \leq rt\} + \{x: \bar \varphi_2(x) \leq r(1-t)\}.
    \]
    Equivalently,
    $$
    C_t^{-}(r)=\left\{x: \varphi_1(x) \leq r t\right\}-\left\{x: \varphi_2(x) \leq r(1-t)\right\} .
    $$
    Set
    \[
        S_r^+ = \bigcup_{0 \leq t \leq 1} C_t^+(r) \quad \text{and} \quad S_r^- = \bigcup_{0 \leq t \leq 1} C_t^-(r).
    \]
    Recall that
    $$
    (\varphi \mathbin\square \psi)(x):=\inf _{a+b=x}\{\varphi(a)+\psi(b)\}.
    $$
    We claim that
    \[
        S_r^+  =\left\{x:\left(\varphi_1 \mathbin\square \varphi_2\right)(x) \leq r\right\}.
    \]
    Indeed, if $x \in C_t^{+}(r)$, then $x=a+b$ for some $a, b$ satisfying
$$
\varphi_1(a) \leq r t, \quad \varphi_2(b) \leq r(1-t) .
$$
Therefore,
$$
\left(\varphi_1 \mathbin\square \varphi_2\right)(x) \leq \varphi_1(a)+\varphi_2(b) \leq r .
$$
This proves one inclusion. Conversely, suppose that
$$
\left(\varphi_1 \mathbin\square \varphi_2\right)(x) \leq r .
$$
The infimum is attained because
$$
a \longmapsto \varphi_1(a)+\varphi_2(x-a)
$$
is lower semicontinuous and tends to $+\infty$ as $|a| \to \infty$. Hence, there are $a, b \in \mathbb{R}^n$ such that
$$
x=a+b, \quad \varphi_1(a)+\varphi_2(b) \leq r .
$$
Set $ t = \frac{\varphi_1(a)}{r}$. Since both base functions are nonnegative, $t \in[0,1]$, and
$$
\varphi_1(a)=r t, \quad \varphi_2(b) \leq r-\varphi_1(a)=r(1-t) .
$$
Thus, $x \in C_t^{+}(r) \subset S_r^+$. Similarly, we have
    \[
        S_r^- = \left\{x:\left(\varphi_1 \mathbin\square \bar \varphi_2\right)(x) \leq r\right\}.
    \]
    Applying Lemma \ref{lem:sublevel-set-representation} to the function $t \mapsto (1-\alpha t)^{1 / \alpha} $, we obtain
    $$
\left\|f \star_\alpha g\right\|_1=\int_0^{\infty}\left|\left\{\varphi_1 \mathbin\square \varphi_2 \leq r\right\}\right|(1-\alpha r)^{1 / \alpha-1} \;\mathrm{d} r=\int_0^{\infty}\left|S_r^+\right|(1-\alpha r)^{1 / \alpha-1} \;\mathrm{d} r
$$
and
$$
\left\|f \star_\alpha \bar{g}\right\|_1=\int_0^{\infty}\left|\left\{\varphi_1 \mathbin\square \bar{\varphi}_2 \leq r\right\}\right|(1-\alpha r)^{1 / \alpha-1} \;\mathrm{d} r=\int_0^{\infty}\left|S_r^{-}\right|(1-\alpha r)^{1 / \alpha-1} \;\mathrm{d} r.
$$
For $\alpha=0$, the same argument holds with $e^{-r}$ in place of $(1-\alpha r)^{1 / \alpha-1}$.
It is enough to prove that
\begin{equation}
    |S_r^+| \leq  \frac{n+1}{2^n}\binom{2 n}{n} |S_r^-|.
\end{equation}
    For each $r>0$, applying Lemma \ref{lem:KM-lp} to $\varphi_1$ and $\varphi_2$ gives

$$
\left|S_r^{+}\right| \leq(n+1) \int_0^1\left|C_t^{+}(r)\right| \mathrm{d} t .
$$
For every $t \in[0,1]$, Litvak's observation yields
$$
\left|C_t^{+}(r)\right| \leq \frac{1}{2^n}\binom{2 n}{n}\left|C_t^{-}(r)\right| .
$$
Applying the right-hand inequality in Lemma \ref{lem:KM-lp} to $\varphi_1$ and $\bar{\varphi}_2$, whose associated family is precisely $C_t^{-}(r)$ , gives
$$
\int_0^1\left|C_t^{-}(r)\right| \mathrm{d} t \leq\left|S_r^{-}\right| .
$$
Consequently,
\[
\left|S_r^{+}\right| \leq(n+1) \int_0^1\left|C_t^{+}(r)\right| \mathrm{d} t \leq \frac{n+1}{2^n}\binom{2 n}{n} \int_0^1\left|C_t^{-}(r)\right| \mathrm{d} t
\leq \frac{n+1}{2^n}\binom{2 n}{n}\left|S_r^{-}\right| . \qedhere
\]
\end{proof}
\begin{remark} \label{remark7}
The factor in \eqref{eq:Litvak-alpha} is not claimed to be sharp. However, it is asymptotically sharp since it converges to $2^n$ as $ \alpha \to 0^-$.
Lemma~\ref{lem: counter-example} shows that every valid constant is at
least \(2^n\).  Indeed, take
\[
 f(x)=(1-\alpha|x|_1)^{1/\alpha}1_{\mathbb R_+^n}(x),
 \qquad g=\bar f,
\]
with the usual exponential interpretation at \(\alpha=0\).  For this
example,
\(\|f\star_\alpha g\|_1=2^n\|f\star_\alpha\bar g\|_1\).
\end{remark}

\begin{prop} \label{eq:improved-diff-alpha}
    Let  $n\geq 2$, and let $\alpha \in (-1/n,0)$. Then,
    \[
        \mathcal{C}(n,\alpha) \leq \min \left\{\frac{n+1}{2^n}\binom{2 n}{n}, \frac{2^n}{\prod_{j=1}^n(1+j \alpha)} \right\}.
    \]
\end{prop}

\begin{proof}
    Let $f \in C_\alpha (\R^n)$, and let $\varphi=\operatorname{base}_\alpha f$. Define $g, h \in C_\alpha\left(\mathbb{R}^n\right)$
$$
\operatorname{base}_\alpha g(x)=\frac{\varphi(2 x)}{2}, \quad \operatorname{base}_\alpha h(x)=\frac{\varphi(-2 x)}{2} .
$$
The convexity of $\varphi$ gives
$$
\operatorname{base}_\alpha\left(g \star_\alpha \bar{h}\right)(x)=\inf _{a+b=x} \frac{\varphi(2 a)+\varphi(2 b)}{2}=\varphi(x) .
$$
Thus $g \star_\alpha \bar{h}=f$. On the other hand, the definition of $\Delta_\alpha f$ gives
$$
\operatorname{base}_\alpha\left(\Delta_\alpha f\right)(x)=\inf _{a+b=x}\left\{\frac{\varphi(2 a)}{2}+\frac{\varphi(-2 b)}{2}\right\},
$$
and hence $\Delta_\alpha f=g \star_\alpha h$. Theorem \ref{prop:litvak-fun} therefore yields
$$
\frac{\left\|\Delta_\alpha f\right\|_1}{\|f\|_1 } \leq \min \left\{\frac{n+1}{2^n}\binom{2 n}{n}, \frac{2^n}{\prod_{j=1}^n(1+j \alpha)} \right\}.
$$
Since $f$ was arbitrary, taking the supremum over all $f \in C_\alpha (\R^n)$ proves the asserted estimate for $\mathcal{C}(n, \alpha)$.
\end{proof}
\begin{remark}
    The first estimate is independent of $\alpha$, whereas the second tends to $2^n$ as $\alpha \rightarrow 0^{-}$. Lemma~\ref{lem: counter-example} shows that the constant is at least $2^n$. Lemma~\ref{lem:computation} shows that the new upper bound is strictly smaller than the previously known one.
\end{remark}
\begin{remark}
In fact, the smallest possible constant of \eqref{eq:Litvak-alpha} is $\mathcal{C}(n,\alpha)$. For completeness, we justify the last assertion.
    Applying the Borell--Brascamp--Lieb inequality to $f \star_\alpha g$, and its reflection gives
    $$
    \left\|f \star_\alpha g\right\|_1 \leq\left\|\Delta_\alpha\left(f \star_\alpha g\right)\right\|_1 .
    $$
    Expanding the convex bases and using the commutativity and associativity of infimal convolution yields
    $$
    \Delta_\alpha\left(f \star_\alpha g\right)=\Delta_\alpha\left(f \star_\alpha \bar{g}\right) .
    $$
    Consequently,
    $$
    \left\|f \star_\alpha g\right\|_1 \leq \mathcal{C}(n, \alpha)\left\|f \star_\alpha \bar{g}\right\|_1 .
    $$
Thus, the optimal constant in \eqref{eq:Litvak-alpha} is at most $\mathcal{C}(n, \alpha)$. Conversely, the construction in the proof of Proposition \ref{eq:improved-diff-alpha} shows that every admissible constant in \eqref{eq:Litvak-alpha} is at least $\mathcal{C}(n, \alpha)$. Hence, the two optimal constants coincide. Therefore, \eqref{eq:Litvak-alpha} when $n=1$ and $ \alpha \in (-1,0]$ is sharp with constant 2.
\end{remark}
One can also obtain, using Theorem \ref{prop:litvak-fun} and Theorem \ref{thm: Plu-Ruz-alpha}, the following inequality in the flavor of Ruzsa's triangle inequality analogue to \cite[Theorem 5.2]{FMZ-24}:
\begin{corollary}
    Let $\alpha \in (-1/n,0]$. Then, for any $f,g,h \in C_\alpha (\R^n)$, \
    \begin{align}
        \| f\|_1 \| f \star_\alpha g \star_\alpha h \|_1
        \leq
        \frac{(n+1)c^n\binom{2n}{n}}{2^n\prod_{j=1}^n(1+j \alpha)^2}  \min \left\{
        \begin{array}{c}
             \| f \star_\alpha {g}  \|_1 \| f \star_\alpha h \|_1
             \\
             \| f \star_\alpha \Bar{g}  \|_1 \| f \star_\alpha h \|_1
             \\
             \| f \star_\alpha {g}  \|_1 \| f \star_\alpha \Bar{h} \|_1
             \\
             \| f \star_\alpha \Bar{g}  \|_1 \| f \star_\alpha \Bar{h} \|_1
        \end{array}
        \right\},
    \end{align}
    where $c$ is an absolute constant such that $c<26$.
\end{corollary}

\begin{proof}
The first estimate follows directly from Theorem \ref{thm: Plu-Ruz-alpha}, which gives
$$
\|f\|_1\left\|f \star_\alpha g \star_\alpha h\right\|_1 \leq \frac{c^n}{\prod_{j=1}^n(1+j \alpha)^2}\left\|f \star_\alpha g\right\|_1\left\|f \star_\alpha h\right\|_1 .
$$
Applying Theorem \ref{prop:litvak-fun}  to the second factor on the right gives
$$
\|f\|_1\left\|f \star_\alpha g \star_\alpha h\right\|_1 \leq \frac{(n+1) c^n}{2^n \prod_{j=1}^n(1+j \alpha)^2}\binom{2 n}{n}\left\|f \star_\alpha \bar{g}\right\|_1\left\|f \star_\alpha h\right\|_1 .
$$
Applying Theorem \ref{prop:litvak-fun} to the third factor instead gives the third estimate.
For the remaining estimate, we first apply Theorem \ref{prop:litvak-fun} to $f$ and $g \star_\alpha h$,
$$
\left\|f \star_\alpha g \star_\alpha h\right\|_1 \leq \frac{n+1}{2^n}\binom{2 n}{n}\left\|f \star_\alpha \bar{g} \star_\alpha \bar{h}\right\|_1 .
$$
Theorem \ref{thm: Plu-Ruz-alpha} applied to $f, \bar{g}, \bar{h}$ then yields
$$
\|f\|_1\|f \star_\alpha \bar{g} \star_\alpha \bar{h}\|_1 \leq \frac{c^n}{\prod_{j=1}^n(1+j \alpha)^2}\|f \star_\alpha \bar{g}\|_1\|f \star_\alpha \bar{h}\|_1 .
$$
Taking the minimum of these four estimates proves the corollary.
\end{proof}

\section{Functional Ruzsa triangle inequality}
\label{sec:triangle}

The following lemma was proved in \cite{FLM-20}. We restate it in the form needed here.
\begin{lemma} \label{FLM-20-formulate}
    Let $\alpha \in (-\frac{1}{n},0]$, and let $f \in C_\alpha (\R^n)$. Then,
    \[
        \|f\|_1 \leq \left(\prod_{j=1}^n \frac{2+j \alpha}{1+j \alpha} \right)\|f\|_2^2.
    \]
\end{lemma}
\begin{proof}
    Let $\rho = \frac{f}{\|f\|_1}$ be an $\alpha$-concave probability density. Recall its R\'enyi entropies of orders 2 and $\infty$,
$$
h_2(\rho)=-\log \int_{\mathbb{R}^n} \rho(x)^2 \; \mathrm d x,
\quad
h_{\infty}(\rho)=-\log \|\rho\|_{\infty} .
$$
Thus, since $\|f\|_\infty = 1$,
\begin{equation}\label{eq:FLM-01}
    h_2(\rho) -h_\infty(\rho)=\left(2 \log \|f\|_1 -\log \|f\|_2^2\right) -\left(\log \|f\|_1-\log \|f\|_{\infty} \right) = \log \left(\frac{\|f\|_1}{\left\|f\right\|_2^2}\right).
\end{equation}
Let
\[
\rho_\alpha (x)= \left(\prod_{j=1}^n(1+j \alpha)\right)\left(1-\alpha |x|_1\right)^{1 / \alpha} {1}_{\mathbb{R}_{+}^n}(x), \quad \rho_0 (x)= e^{-|x|_1} {1}_{\mathbb{R}_{+}^n}(x).
\]
Thus,
\[
    \|\rho_\alpha\|_2^2 = \frac{\left(\prod_{j=1}^n(1+j \alpha)\right)^2}{\prod_{j=1}^n(2+j \alpha)}, \quad \| \rho_\alpha \|_\infty = \prod_{j=1}^n(1+j \alpha).
\]
Therefore,
\begin{equation} \label{eq:FLM-02}
    h_2(\rho_\alpha) -h_\infty(\rho_\alpha) = \log \left(\frac{\prod_{j=1}^n(2+j \alpha)}{\prod_{j=1}^n(1+j \alpha)} \right).
\end{equation}
Using \cite[Corollary 7.1]{FLM-20}, we obtain that
    \[
        h_2(\rho) -h_\infty(\rho) \leq h_2(\rho_\alpha) -h_\infty(\rho_\alpha).
    \]
    Substituting \eqref{eq:FLM-01} and \eqref{eq:FLM-02}, the proof is complete.
\end{proof}

\begin{theorem} \label{thm:ruza-functional}
    Let $\alpha \in (-\frac{1}{n},0]$, and let $f,g,h  \in C_\alpha (\R^n)$. Then,
    \[
        \|f\|_1 \|g \star_\alpha \Bar h\|_1 \leq \frac{1}{2^n} \prod_{j=1}^n\left(\frac{2+j \alpha}{1+j \alpha}\right)^2 \|f \star_\alpha \Bar g \|_1 \|f \star_\alpha \Bar h \|_1.
    \]
\end{theorem}
\begin{proof}
Throughout the proof, expressions of the form
$$
(1-\alpha t)^{1 / \alpha} \quad \text { and } \quad \left(1-\frac{\alpha}{2} t\right)^{2 / \alpha}
$$
are understood as $e^{-t}$ when $\alpha=0$.
We claim that
\begin{equation} \label{eq:claim-Ruzsa-1}
    (1-\alpha(u+s))^{1 / \alpha}(1-\alpha(u+t))^{1 / \alpha} \geq(1-\alpha u)^{2 / \alpha}\left(1-\frac{\alpha}{2}(s+t)\right)^{2 / \alpha}.
\end{equation}
for every $u, s, t \geq 0$. For $\alpha=0$, both sides are the same. Suppose that $\alpha<0$. Since
$$
1-\alpha(u+s) \leq(1-\alpha u)(1-\alpha s)
$$
and $1 / \alpha<0$, it follows that
$$
(1-\alpha(u+s))^{1 / \alpha} \geq(1-\alpha u)^{1 / \alpha}(1-\alpha s)^{1 / \alpha} .
$$
Similarly,
$
(1-\alpha(u+t))^{1 / \alpha} \geq(1-\alpha u)^{1 / \alpha}(1-\alpha t)^{1 / \alpha} .
$
Multiplying them together and using the arithmetic-geometric mean inequality give
\begin{align}
(1-\alpha(u+s))^{1 / \alpha}(1-\alpha(u+t))^{1 / \alpha}
&\geq(1-\alpha u)^{2 / \alpha}((1-\alpha s)(1-\alpha t))^{1 / \alpha}
\\
&\geq (1-\alpha u)^{2 / \alpha} \left(1-\frac{\alpha}{2}(s+t)\right)^{2 / \alpha} ,
\end{align}
verifying the claim \eqref{eq:claim-Ruzsa-1}.

Let $x \in \R^n$. We fix $a, b, c \in \mathbb{R}^n$. By the definition of \eqref{def:alpha-sum-original}, we have
$$
\left(f \star_\alpha \bar{g}\right)(a-b) \geq\left(1-\alpha\left( \operatorname{ base }_\alpha f(a)+ \operatorname{ base }_\alpha g(b)\right)\right)^{1 / \alpha},
$$
and
$$
    (f \star_\alpha \bar{h})(a-c) \geq\left(1-\alpha\left(\operatorname{ base }_\alpha f(a)+\operatorname{ base }_\alpha h(c)\right)\right)^{1 / \alpha}.
$$
Applying the claim \eqref{eq:claim-Ruzsa-1} with
$
u=\operatorname { base }_\alpha f(a),  s=\operatorname { base }_\alpha g(b),  t=\operatorname { base }_\alpha h(c),
$ we obtain
\[
    \left(f \star_\alpha \bar{g}\right)(a-b)(f \star_\alpha \bar{h})(a-c)
    \geq f^2(a)\left(1-\frac{\alpha}{2}\left(\operatorname{base}_\alpha g(b)+\operatorname{base}_\alpha h(c)\right)\right)^{2 / \alpha} .
\]
Integrating both sides with respect to $a$, making a change of variables $u = a-b$, and taking the supremum over all $b,c$ such that $x = b-c$,
\begin{align}
    \int \left(f \star_\alpha \bar{g}\right) (u) (f \star_\alpha \bar{h})(u+x) \;\mathrm d u
    & =\int \left(f \star_\alpha \bar{g}\right) (a-b) (f \star_\alpha \bar{h})(a-c) \;\mathrm d a
    \\
    &\geq \|f\|_2^2\sup_{b-c = x}\left(1-\frac{\alpha}{2}\left(\operatorname{base}_\alpha g(b)+\operatorname{base}_\alpha h(c)\right)\right)^{2 / \alpha}
    \\
    &=
    \|f\|_2^2\left(1-\frac{\alpha}{2}\inf_{b-c = x} \left(\operatorname{base}_\alpha g(b)+\operatorname{base}_\alpha h(c)\right)\right)^{2 / \alpha}
    \\
    &=
    \|f\|_2^2\left(1-\frac{\alpha}{2}\left(\operatorname{base}_\alpha (g \star_\alpha \bar h )\right)(x)\right)^{2 / \alpha} .
\end{align}
Integrating in $x$ and applying Tonelli's theorem gives
\begin{align}
    \left\|f \star_\alpha \bar{g}\right\|_1\|f \star_\alpha \bar{h}\|_1
    &= \int \int \left(f \star_\alpha \bar{g}\right) (u) (f \star_\alpha \bar{h})(u+x) \;\mathrm d x \;\mathrm d u
    \\
    &\geq \|f\|_2^2 \int \left(1-\frac{\alpha}{2} \operatorname{base}_\alpha\left(g \star_\alpha \bar{h}\right)(x)\right)^{2 / \alpha} \mathrm{d} x.
\end{align}
Using Lemma \ref{FLM-20-formulate} and Theorem \ref{thm:monotonicity-alpha-base} with $\alpha_1=\alpha$ and $\alpha_2=\alpha / 2$,
\begin{align}
    \|f\|_2^2 \int \left(1-\frac{\alpha}{2} \operatorname{base}_\alpha\left(g \star_\alpha \bar{h}\right)(x)\right)^{2 / \alpha} \mathrm{d} x
    \geq
    \left(\prod_{j=1}^n \frac{1+j \alpha}{2+j \alpha} \right) \|f\|_1 2^n \prod_{j=1}^n \frac{1+j \alpha}{2+j \alpha}\left\|g \star_\alpha \bar{h}\right\|_1.
\end{align}
Combining the last two estimates yields the desired inequality.
\end{proof}

\begin{remark}
For $\alpha\in(-\frac{1}{n},0]$, set $f(x)=\bigl(1-\alpha|x|_1\bigr)^{1/\alpha}\cdot {1}_{\mathbb{R}^n_+}(x),$
    and in the limit case $\alpha=0$, let
    \(
    f(x)=e^{-|x|_1} \cdot {1}_{\mathbb{R}^n_+}(x).
    \)
Consider  $f= \bar g = \bar h $. Then,
\[
\|f \|_1 \|g \star_\alpha \bar h\|_1 =\|f \|_1 \|\bar f \star_\alpha f \|_1 = 2^n \|f \star_\alpha f\|_1 \|f \star_\alpha f\|_1 = 2^n \|f \star_\alpha \bar g\|_1 \|f \star_\alpha \bar h\|_1,
\]
where the second equality uses Lemma~\ref{lem: counter-example}. Thus, any functional analogue of \eqref{eq:Ruzsa-ineq-wo-A} in the range \(\alpha\in(-1/n,0]\) must have multiplicative constant at least $2^n$.
\end{remark}

\begin{corollary} \label{cor:ruzvol1-func}
    Let $\alpha \in (-\frac{1}{n},0]$, and let $f,g,h  \in C_\alpha (\R^n)$. Then
    \[
        \|f\|_1 \|g \star_\alpha  h\|_1 \leq \frac{4^n}{\prod_{j=1}^n (1+j \alpha)^2} \|f \star_\alpha  g \|_1 \|f \star_\alpha  h \|_1.
    \]
\end{corollary}

\begin{proof}
    Using Corollary \ref{cor:01}, and the definition \eqref{def:alpha-sum-original} gives
$$
\begin{aligned}
\|f\|_1\left\|g \star_\alpha h\right\|_1 & \leq \frac{1}{\prod_{j=1}^n (1+j \alpha)^2}\left\|B_\alpha f\right\|_1\left\|B_\alpha\left(g \star_\alpha h\right)\right\|_1 \\
& =\frac{1}{\prod_{j=1}^n (1+j \alpha)^2}\left\|B_\alpha f\right\|_1\left\|B_\alpha g \star_0 B_\alpha h\right\|_1 .
\end{aligned}
$$
Apply Theorem \ref{thm:ruza-functional} with $\alpha=0$ to the three log-concave functions $B_\alpha f, B_\alpha g$, and $\overline{B_\alpha h}$. We obtain
$$
\left\|B_\alpha f\right\|_1\|B_\alpha g \star_0 B_\alpha h\|_1 \leq 2^n\|B_\alpha f \star_0 \overline{B_\alpha g}\|_1\left\|B_\alpha f \star_0 B_\alpha h\right\|_1 .
$$
Theorem \ref{prop:litvak-fun} with $\alpha=0$, the definition \eqref{def:alpha-sum-original} and \eqref{eq:02} yield
$$
\left\|B_\alpha f \star_0 \overline{B_\alpha g}\right\|_1 \leq 2^n\left\|B_\alpha f \star_0 B_\alpha g\right\|_1  = 2^n \left\|B_\alpha (f \star_\alpha g)\right\|_1 \leq 2^n\|f \star_\alpha g\|_1.
$$
Combining the preceding inequalities gives the desired inequality.
\end{proof}

\appendix

\section{}\label{appendixtech}
\begin{lemma} \label{upper-bound-c}
    Let $c_n$ be the smallest possible number satisfying that for any convex bodies $A,B,C$ in $\R^n$, one has
\begin{equation}
    |A| |A+B+C| \leq c_n |A+B| |A+C|.
\end{equation}
Then, $c_n (4^n +n) \binom{2n}{n}<26^n$.
\end{lemma}
\begin{proof}
Note that it was proved in \cite{FMZ-24} that
$$
c_n \leq\left(\frac{1+\sqrt{5}}{2}\right)^n .
$$
Observe that
\[
     4^n=\sum_{k=0}^{2 n}\binom{2 n}{k} \geq\binom{ 2 n}{n-1}+\binom{2 n}{n}+\binom{2 n}{n+1} =\left(1+\frac{2 n}{n+1}\right)\binom{2 n}{n} \geq 2 \binom{2n}{n}.
\]
    Since $n \leq 4^n$,
    we obtain
$$
\left(4^n+n\right)\binom{2 n}{n} \leq\left(2 \cdot 4^n\right) \frac{4^n}{2}=16^n .
$$
Therefore,
$$
c_n\left(4^n+n\right)\binom{2 n}{n} \leq\left(\frac{1+\sqrt{5}}{2}\right)^n 16^n
 =(8(1+\sqrt{5}))^n .
$$
Finally, since $\sqrt{5}<9 / 4$,
\[
8(1+\sqrt{5})<8\left(1+\frac{9}{4}\right)=26 . \qedhere
\]
\end{proof}
\begin{lemma} \label{lem:computation}
    Let $n \geq 1$ and $-1 / n<\alpha<0$. Then
$$
\min \left\{\frac{n+1}{2^n}\binom{2 n}{n}, \frac{2^n}{\prod_{j=1}^n(1+j \alpha)}\right\}<\frac{1}{2^{n+1 / \alpha}}\binom{2 n}{n} .
$$
\end{lemma}
\begin{proof}
    Since $1 / \alpha<-n$ and $n+1 \leq 2^n$,
$$
(n+1) 2^{1 / \alpha}<(n+1) 2^{-n} \leq 1 .
$$
Thus,
$$
\frac{n+1}{2^n}\binom{2 n}{n}=(n+1) 2^{1 / \alpha} \frac{1}{2^{n+1 / \alpha}}\binom{2 n}{n}<\frac{1}{2^{n+1 / \alpha}}\binom{2 n}{n} .
$$
The result follows since the minimum is at most its first term.
\end{proof}

\begin{lemma}
    \label{lem: counter-example}
    For every $\alpha \in(-1 / n, 0]$, there exists $f \in C_\alpha\left(\mathbb{R}^n\right)$ such that
    \[
        \|f \star_\alpha \Bar{f} \|_1= 2^n \|f \star_\alpha f\|_1 = 2^n\|f\|_1,
    \]
    where $ \Bar{f}(x) = f(-x)$.
\end{lemma}

\begin{proof}
    Set $f(x)=\bigl(1-\alpha|x|_1\bigr)^{1/\alpha}\cdot {1}_{\mathbb{R}^n_+}(x),$
    and in the limit case $\alpha=0$, let
    \(
    f(x)=e^{-|x|_1} \cdot {1}_{\mathbb{R}^n_+}(x).
    \)
    We provide the proof for $\alpha<0$, and the case $\alpha=0$ follows by the same argument.
    Using \eqref{eq:explicit-form-some-alpha-sum}, we have
    \begin{align*}
        (f \star_\alpha f)  (x)
        = \sup_{x = y+z} (f^\alpha (y) + f^\alpha(z) - 1)^{1/\alpha} = f(x).
    \end{align*}
    The last equality follows because it is enough to consider $y, z \in \mathbb{R}_{+}^n$; if either function value vanishes, the expression is zero by convention.

    Now, using \eqref{eq:explicit-form-some-alpha-sum}, we obtain
    \begin{align*}
        (f \star_\alpha \bar f) (x) = \sup_{x = y+z} (f^\alpha (y) + f^\alpha(-z) - 1)^{1/\alpha}.
    \end{align*}
    Note that the expression is maximized when $ y,-z \in \R^n_+$.
    By setting $ y_i = \max \{x_i,0\}\geq 0$ and $z_i = \min \{x_i,0\}\leq 0$, we have $(f \star_\alpha \bar f) (x) \geq (1-\alpha |x|_1)^{1/\alpha}.$

    On the other hand, for any $y,-z \in \R^n_+$ such that $x = y+z$, we have $|y|_1 + |z|_1 \geq |x|_1$ and hence
     \[
        f^\alpha(y)+f^\alpha(-z)-1=1-\alpha\left(|y|_1+|z|_1\right) \geq 1 - \alpha |x|_1.
     \]
    Thus, taking the sign of \(\alpha\) into account, we obtain that $(f \star_\alpha \bar f) (x) \leq (1-\alpha |x|_1)^{1/\alpha} $ and so
    \[
        (f \star_\alpha \bar f) (x) = (1-\alpha |x|_1)^{1/\alpha}.
    \]
    Therefore,
    \[
        \|f \star_\alpha \bar f \|_1 = \int (1-\alpha |x|_1)^{1/\alpha} \;\mathrm d x  = 2^n \|f\|_1 = 2^n \|  f \star_\alpha f \|_1. \qedhere
    \]
\end{proof}

\section*{Acknowledgements}

This work was initiated during the ``Harmonic Analysis and Convexity'' program (NSF Grant DMS-1929284) and completed during the
``Synergies between Geometry, Probability, and Computation in High Dimensions'' program (NSF Grant DMS-2424556)
at the Institute for Computational and Experimental Research in Mathematics (ICERM).

The authors also thank Universit\'e Gustave Eiffel and the University of Warsaw for their hospitality. Parts of the work were completed during visits by the second and fourth authors to Matthieu Fradelizi at Universit\'e Gustave Eiffel and by the second and third authors during the “Inequalities and Log-Concavity” program at the University of Warsaw.

During the preparation of this paper, ChatGPT was used as an auxiliary tool to assist with refining the proof arguments. The final statements here are those of the authors, who have verified them and take full responsibility for the content of the paper.

A. Manui was supported by the Thailand Development and Promotion of Science and Technology talent project and the Chateaubriand Fellowship of the Office for Science \& Technology of the Embassy of France in the United States.

A. Manui and A. Zvavitch were supported by the NSF Grant DMS-2247771 and the BSF Grant 2018115.

B. Zawalski was supported in part by NSF Grants DMS-1900008 and DMS-2247771.
% \renewcommand*{\bibfont}{\footnotesize}
% \printbibliography

\bibliographystyle{siam}
% \bibliography{references-GeneralCase}
\end{document}